\documentclass[12pt]{amsart}%
\usepackage{amsfonts}
\usepackage{amsmath}
\usepackage{amssymb}
\usepackage{amsthm}
\usepackage{mathrsfs}
\usepackage{graphicx}%
\usepackage{amsmath}
\usepackage{setspace}
\usepackage{hyperref}
\usepackage{cite}
\usepackage{amsfonts}
\usepackage{algorithmic}
\usepackage{textcomp}
\newcommand{\R}{\mathbb R}
\newcommand{\AF}{\mathcal A_F}

\DeclareMathOperator{\distg}{dist_g}
\allowdisplaybreaks[4]
\makeatletter
\@addtoreset{equation}{section}
\makeatother

\newtheorem{theorem}{Theorem}[section]

\newtheorem{lemma}[theorem]{Lemma}

\newtheorem{proposition}[theorem]{Proposition}

\makeatletter
\def\@makefnmark{}
\makeatother

\begin{document}

\title[Multiplicity One for Strictly Stable Anisotropic Allen-Cahn Minimal Hypersurfaces]{Multiplicity One for Strictly Stable Anisotropic Allen-Cahn   Minimal Hypersurfaces}

\author{Lu Chen}
\address[Lu Chen]{Key Laboratory of Algebraic Lie Theory and Analysis of Ministry of Education, School of Mathematics and Statistics, Beijing Institute of Technology, Beijing
100081, PR China;  Tangshan Research Institute, Beijing Institute of Technology, Tangshan 063000, PR China}
\email{chenlu5818804@163.com}

\author{Jiali Lan}
\address[Jiali Lan]{Key Laboratory of Algebraic Lie Theory and Analysis of Ministry of Education, School of Mathematics and Statistics, Beijing Institute of Technology, Beijing
100081, PR China}
\email{17636268505@163.com}


\thanks{The first author was partly supported by the  National Natural Science Foundation of China (No. 12622108),  a grant from Beijing Institute of Technology (No.2022CX01002). }

\begin{abstract}
 In this paper, we study the multiplicity of smooth limit interfaces for the anisotropic Allen-Cahn equation on closed Riemannian manifolds. We prove that an isolated smooth limit interface has multiplicity one whenever it is a closed, embedded, two-sided, strictly \(F\)-stable \(F\)-minimal hypersurface. The proof uses anisotropic barriers and a sliding argument to confine the nodal sets to a neighborhood of width \(O(\varepsilon)\). We then classify the nonconstant rescaled limits as one-dimensional solutions and obtain an energy upper bound near the limiting hypersurface. Exponential decay controls the remaining energy, yielding the multiplicity-one conclusion. 

\end{abstract}

\maketitle

 \section{Introduction and Main Results}



The anisotropic Allen-Cahn equation describes phase transitions with
direction-dependent surface energy. The corresponding minimal hypersurfaces are limits of nodal sets of
solutions, whose energy concentrates on each connected component of the hypersurface as an integer
multiple of its area.  In this paper, we study the multiplicity of smooth limit interfaces.

 Let $(M^{n},g)$ be a closed Riemannian manifold with $n\geq2$. We consider a smooth even function:
\begin{equation}\label{e1.4}F : UT(M) \to (0, \infty), \qquad F(x, v) = F(x, -v),\end{equation}
defined on the unit tangent bundle, and extend \(F\) 1-homogeneously to a function that belongs to $C^4(TM\setminus\{0\})$:
\begin{equation}\label{e1.5}F : TM \to [0, \infty), \qquad F(x, \lambda v) = \lambda F(x, v)\text{ for all }  \lambda \ge 0.\end{equation}
  Set
\begin{equation}\label{e1.5}
 \Phi(x,\xi)=\frac{1}{2}F(x,\xi)^{2},\qquad A(x,\xi)=D_{\xi}\Phi(x,\xi),
\end{equation}
where $A(x,0)=0$. It follows from the homogeneity of $F$ that
\begin{equation}\label{e2.1}
A(x,\xi)\cdot\xi=F(x,\xi)^{2},
\qquad A(x,t\xi)=tA(x,\xi)\ \forall t\in\R.
\end{equation}
 We assume that there are constants $0<\lambda\leq\Lambda$ satisfy
\begin{equation}\label{e1.6}
 \lambda|\eta|_{g}^{2}\leq D^{2}_{\xi\xi}\Phi(x,\xi)[\eta,\eta]\leq\Lambda|\eta|_{g}^{2}
\end{equation}
for every $x\in M$, $\xi\neq0$ and $\eta\in T_{x}M$. Differentiating the Euler identity gives
\[
 D_{\xi\xi}^{2}\Phi(x,\xi)[\xi,\xi]=2\Phi(x,\xi)=F(x,\xi)^{2},
\]
which implies that
\begin{equation}\label{e2.2}
 \lambda|\xi|_{g}^{2}\leq F(x,\xi)^{2}\leq\Lambda|\xi|_{g}^{2}.
\end{equation}
The   anisotropic Allen-Cahn energy of a function \(u : M \to \mathbb{R}\) is
\begin{equation}\label{e1.1}
 E_{\varepsilon,F}(u;M)=\int_{M}\left(\frac{\varepsilon}{2}F(x,\nabla u)^{2}+\frac{W(u)}{\varepsilon}\right)\,\mathrm{d}V_{g},
\end{equation}
where
\begin{equation}\label{e1.2}
 W(s)=\frac{1}{4}(1-s^{2})^{2}.
\end{equation}
Its Euler-Lagrange equation is
\begin{equation}\label{e1.3}
 -\varepsilon^{2}\operatorname{div}_{g} A(x,\nabla u)+W'(u)=0.
\end{equation}
When $F(x,\xi)=|\xi|_{g}$, equation \eqref{e1.3}  reduces to the classical Allen-Cahn equation.

The anisotropic surface area of an embedded hypersurface \(\Sigma  \subset M^n\) is defined by
\[\mathcal{A}_F(\Sigma) := \int_{\Sigma} F(x, \nu_\Sigma) \, d\mathcal H_g^{n-1}.\]
 A critical point of \(\AF\) is called an \(F\)-minimal
hypersurface. The relation between the diffuse energy $E_{\varepsilon,F}$ and the surface
energy $\mathcal{A}_F$ can be expressed in terms of \(\Gamma\)-convergence.
Bouchitt\'e~\cite{Bouchitte1990} established this variational limit
in a general class of inhomogeneous anisotropic two-phase models.
With the normalization in \eqref{e1.1}, the limiting
surface energy is \(\sigma_W\AF\), where
\[
 \sigma_W=\int_{-1}^{1}\sqrt{2W(s)}\,ds.
\]
Cicalese, Nagase and Pisante~\cite{CicaleseNagasePisante2010}
proved the Gibbs-Thomson relation for inhomogeneous anisotropic
phase transitions in a mass-constrained minimization problem, relating
the limiting chemical potential to the anisotropic mean curvature
of the interface.

For general critical points, to understand the interface limit, one also needs control over both energy concentration and the geometry of the limiting measure.
In particular, \(\Gamma\)-convergence alone does not determine the
multiplicity of a limit interface. In the isotropic setting,
Hutchinson and Tonegawa~\cite{HutchinsonTonegawa2000} established
convergence to stationary integral varifolds, while Tonegawa and
Wickramasekera~\cite{TonegawaWickramasekera2012} proved regularity
of the limit under a stability assumption. Building on these results,
Guaraco \cite{Guaraco2018} constructed closed minimal
hypersurfaces from min-max critical points of the Allen-Cahn energy.

Min-max methods have also been developed for anisotropic surface
energies. De Philippis and De Rosa~\cite{DePhilippisDeRosa2024}
constructed anisotropic minimal and constant anisotropic mean
curvature surfaces in closed three-dimensional Riemannian manifolds,
smooth away from at most one point. De Philippis, De Rosa and
Li~\cite{DePhilippisDeRosaLi2024} subsequently established existence
and regularity of anisotropic min-max hypersurfaces in closed
manifolds of   dimension $n\geq3$. In particular, their
result gives a smooth anisotropic minimal surface in every closed
three-dimensional manifold.

A phase-transition approach to this construction was developed by
De Rosa and Pigati~\cite{DeRosaPigati2025}. They proved a Modica-type
gradient estimate for the anisotropic Allen-Cahn equation and used
estimates for stable solutions to obtain concentration on
\(F\)-stationary integral varifolds for sequences with bounded energy
and a uniform Morse index bound.
This yields an anisotropic min-max construction through the
Allen-Cahn energy and leads to the question of the multiplicity
of the resulting interfaces.

The multiplicity represents the number of transition layers that accumulate on the same hypersurface.  Near a smooth connected
limit interface \(\Gamma\), we consider energy convergence of the form
\begin{equation}\label{e1.11}
 \left(\frac{\varepsilon_{j}}{2}F(x,\nabla u_{j})^{2}+\frac{W(u_{j})}{\varepsilon_{j}}\right)\mathrm{d}V_{g}
 \rightharpoonup
 m\sigma_{W}F(x,\nu_{\Gamma})\mathcal{H}^{n-1}_{g}\llcorner\Gamma,  \qquad m\in\mathbb N_{>0}.
\end{equation}

Although the nodal sets converge to the same $\Gamma$ in the Hausdorff sense, this convergence alone does not distinguish a single transition layer from multiple layers accumulating on the same hypersurface. Thus, even when the interface and energy
convergence are given, it remains to determine geometric conditions
which force \(m=1\).

In the isotropic case, Chodosh and
Mantoulidis~\cite{ChodoshMantoulidis2020} proved curvature and sheet
separation estimates on three-manifolds. For generic metrics, they
obtained multiplicity one for limits of Allen-Cahn solutions with
uniformly bounded energy and Morse index. Guaraco, Marques and
Neves~\cite{Guaraco2022} proved a local multiplicity-one theorem in
arbitrary dimensions under strict stability of the limit interface,
without stability or index assumptions on the approximating solutions.
Here we extend this strict-stability result to anisotropic energies
whose integrands may depend on position.

We recall the definition of stability before stating our main result.
Let \(\Gamma\) be a smooth, closed, two-sided \(F\)-minimal
hypersurface, and fix a unit normal \(\nu_\Gamma\). For
\(\varphi\in C^\infty(\Gamma)\), let
\[
 \Gamma_t^\varphi
 =\bigl\{\exp_x(t\varphi(x)\nu_\Gamma(x)):x\in\Gamma\bigr\},
 \qquad
 Q_F(\varphi)
 =\left.\frac{d^2}{dt^2}\right|_{t=0}\AF(\Gamma_t^\varphi).
\]
We say that \(\Gamma\) is \emph{\(F\)-stable} if
\(Q_F(\varphi)\ge0\) for every \(\varphi\), and
\emph{strictly \(F\)-stable} if \(Q_F(\varphi)>0\) for every
 \(\varphi\neq0\). Under the ellipticity assumption above, strict
stability on the closed hypersurface is equivalent to positivity
of the first eigenvalue of the self-adjoint operator associated
with \(Q_F\).

Write
\[
 N_r(\Gamma)=\{x\in M:\distg(x,\Gamma)<r\}.
\]
Our main result is the following.

\begin{theorem}\label{thm:main}
Let \(\Omega\subset M\) be open, let \(\varepsilon_j\to0\), and
let \(u_j\) be weak solutions of \eqref{e1.3} in
\(\Omega\) with \(\varepsilon=\varepsilon_j\) and \(|u_j|\le1\).
Suppose that \(\Gamma\subset\subset\Omega\) is a smooth, closed, connected,
embedded, two-sided \(F\)-minimal hypersurface. Assume that, for
some sufficiently small \(\rho>0\) with
\(\overline{N_{2\rho}(\Gamma)}\subset\Omega\),
\begin{equation}\label{e1.10}
 \{u_j=0\}\cap N_{2\rho}(\Gamma)\longrightarrow\Gamma
\end{equation}
in the Hausdorff distance, and that
\eqref{e1.11} holds as Radon measures in
\(N_{2\rho}(\Gamma)\). If \(\Gamma\) is strictly \(F\)-stable,
then \(m=1\).
\end{theorem}

The theorem applies to each connected component satisfying the above
assumptions in every dimension \(n\ge2\).
The stability assumption concerns the limiting hypersurface;
no stability or Morse index bound is imposed on \(u_j\).
For global weak solutions on the closed manifold, the bound
\(|u_j|\le1\) follows from the equation, while in the local statement it is assumed explicitly.

Using strict \(F\)-stability and the second variation of
\(\AF\), we construct a foliation near \(\Gamma\) with
the required sign of anisotropic mean curvature.
We construct barriers from the associated anisotropic
signed distance functions \(d_s\), which satisfy
\(F(x,\nabla d_s)=1\).
To compare these barriers with the solutions, we use
a finite-difference argument that does not require
\(A(x,\xi)\) to be differentiable at \(\xi=0\).

Strict \(F\)-stability first gives foliations on both sides of
\(\Gamma\), with the required sign of anisotropic mean curvature.
The associated barriers restrict the nodal sets to an
\(O(\varepsilon_j)\) tubular neighborhood.  Rescaling by
\(\varepsilon_j\)  around points \(p_j\in\Gamma\).
After passing to a subsequence with \(p_j\to p\),
the rescaled solutions converge locally to a solution of
\[
-\operatorname{div} A(p,\nabla v)+W'(v)=0
\qquad\text{in }\mathbb R^n.
\]
In this limiting equation, the vector field \(A(p,\xi)\)
depends only on \(\xi\).
The preceding estimate on the nodal sets implies that
the zero set of \(v\) lies between two parallel
hyperplanes.

A sliding argument shows that every nonconstant limit
is a one-dimensional solution connecting \(-1\) and \(1\).
We use this classification to estimate the energy in
\(N_{R\varepsilon_j}(\Gamma)\). The energy in the remaining
region \(N_\rho(\Gamma)\setminus N_{R\varepsilon_j}(\Gamma)\)
is estimated using exponential decay.
Combining these estimates and letting \(R\to\infty\),
we obtain
\[
\limsup_{j\to\infty}
E_{\varepsilon_j,F}(u_j;N_\rho(\Gamma))
\le \sigma_W\AF(\Gamma).
\]
Comparison with \eqref{e1.11} gives \(m\le1\),
and hence \(m=1\).

\medskip
This paper is organized as follows.
In Section~2, we establish comparison principles and obtain
regularity and decay estimates for solutions of the anisotropic
Allen-Cahn equation.
In Section~3, we use strict \(F\)-stability to construct suitable
barrier functions and show that the nodal sets lie in an
\(O(\varepsilon_j)\) tubular neighborhood of \(\Gamma\).
In Section~4, we prove that
bounded entire solutions whose zero sets are contained between
two parallel hyperplanes are one-dimensional.
Finally, in Section~5, we apply this result to the rescaled
solutions and combine the resulting energy estimate with
exponential decay to complete the proof of
Theorem~\ref{thm:main}.

\section{preliminaries}
\begin{lemma}\label{lem:flux}
For $x\in M$ and $p,q\in T_{x}M$, set
\[
 \mathcal{B}_{\Phi}(x;p,q)
 =
 \Phi(x,p)-\Phi(x,q)-A(x,q)\cdot(p-q).
\]
Then, we have
\begin{equation}\label{eq:bregman}
 \frac{\lambda}{2}|p-q|_{g}^{2}
 \leq
 \mathcal{B}_{\Phi}(x;p,q)
 \leq
 \frac{\Lambda}{2}|p-q|_{g}^{2},
\end{equation}
and
\begin{equation} \label{e2.3}
\bigl(A(x,p)-A(x,q)\bigr)\cdot(p-q)\geq\lambda|p-q|_{g}^{2},\qquad
 |A(x,p)-A(x,q)|_{g}\leq\Lambda|p-q|_{g},
\end{equation}
where $\lambda$ and $\Lambda$ are constants satisfy \eqref{e1.6}.
Moreover, there is a symmetric measurable matrix $B(x,p,q)$ such that
\begin{equation}\label{e2.5}
 A(x,p)-A(x,q)=B(x,p,q)(p-q),
 \qquad
 \lambda I\leq B(x,p,q)\leq\Lambda I.
\end{equation}
 Finally, in local coordinates,  there exists a constant $C$ such that
\begin{equation}\label{eq:base-A}
 |A(x,p)-A(y,p)|\leq C|x-y||p|.
\end{equation}
\end{lemma}
\begin{proof}
Let $\gamma(t)=q+t(p-q)$.   The uniform Hessian bound in \eqref{e1.6} and the continuity of
\(A(x,\cdot)\) at the origin imply that
\(t\mapsto A(x,\gamma(t))\) is Lipschitz, hence absolutely continuous,
on \([0,1]\). For the integrals below, we set
\(D_{\xi\xi}^2\Phi(x,0)=\lambda I\) by convention. Therefore, we have
\[
\Phi(x,p)-\Phi(x,q)
=\int_0^1 \frac{d}{dt}\Phi(x,\gamma(t))\,dt
=\int_0^1 A(x,\gamma(t)) (p-q)\,dt,
\]
which implies that
\[\begin{split}
\mathcal B_\Phi(x;p,q)&=\int_0^1 \bigl[A(x,\gamma(t))-A(x,q)\bigr] (p-q)\,dt\\
&=\int_0^1 \left(\int_0^t \frac{d}{ds}A(x,\gamma(s))\,ds\right) (p-q)\,dt\\
&=\int_0^1  \left(\int_0^t D_{\xi\xi}^2\Phi(x,\gamma(s))(p-q)\,ds\right) (p-q)\,dt\\
&=\int_0^1(1-s)D_{\xi\xi}^2\Phi(x,\gamma(s))[p-q,p-q]\,ds,
\end{split}\]
where we used the Fubini theorem in the last equality.
Then, a direct calculation using \eqref{e1.6} gives \eqref{eq:bregman}.  Similarly,
\[
 A(x,p)-A(x,q)=\int_{0}^1\frac{d}{dt}A(x,\gamma(t))
 =
 \int_{0}^{1}
 D_{\xi\xi}^{2}\Phi(x,\gamma(t)) (p-q)\,\mathrm{d}t.
\]
Therefore, the lower bound and the  upper bound  in \eqref{e1.6} gives \eqref{e2.3}. Equivalently, we can write
\[A(x,p)-A(x,q)=B(x,p,q)(p-q),\]
where $B(x,p,q)=\int_{0}^{1} D_{\xi\xi}^{2}\Phi(x,\gamma(t))\,\mathrm{d}t$. From \eqref{e1.6} we deduce that $B(x,p,q)$ is a symmetric measurable matrix satisfying
\[\lambda I\leq B(x,p,q)\leq\Lambda I.\]
Finally, let \(x,y\) lie in a fixed relatively compact
coordinate ball and set \(\eta(t)=x+t(y-x)\)
in these coordinates.
Since \(A\) is one-homogeneous in \(p\) and smooth
for \(p\ne0\), compactness gives
\[
 |\partial_x A(x,p)|
 =
 |p|\,|\partial_x A(x,p/|p|)|
 \le C|p|,
 \qquad p\ne0.
\]
Therefore, we obtain
\[
|A(y,p)-A(x,p)|
\le\int_0^1
|\nabla_x A(\eta(t), p)|\,|y-x|dt
\le
C |p|\,|y-x|.
\]
The estimate is immediate when \(p=0\).
\end{proof}

\begin{lemma}\label{lem:range}
Every global weak solution $u\in H^1(M)\cap L^4(M)$ of \eqref{e1.3} satisfies
\begin{equation}\label{e2.7}
 |u|\leq1\qquad\text{a.e. on }M.
\end{equation}
\end{lemma}

\begin{proof}
A function $u\in H^{1}(M)\cap L^{4}(M)$ is a global weak solution of \eqref{e1.3} if
\begin{equation}\label{e2.6}
 \varepsilon^{2}\int_{M}A(x,\nabla u)\cdot\nabla\varphi\,\mathrm{d}V_{g}
 +\int_{M}W'(u)\varphi\,\mathrm{d}V_{g}=0
\end{equation}
for every $\varphi\in H^{1}(M)\cap L^{4}(M)$. In particular, we  take $\varphi=(u-1)_{+}$, then from \eqref{e2.1} we obtain
\[
 0=\varepsilon^{2}\int_{\{u>1\}}F(x,\nabla u)^{2}\,\mathrm{d}V_{g}
 +\int_{\{u>1\}}(u^{3}-u)(u-1)\,\mathrm{d}V_{g}\geq0,
\]
the set \(\{u>1\}\) has measure zero, which implies \(u\le 1\) a.e.\ on \(M\). Similarly, taking the test function $\varphi=-(u+1)_{-}$  gives $u\geq-1$ a.e.\ on \(M\).
\end{proof}

By the standard \(C^{1,\alpha}\) regularity theory for uniformly elliptic quasilinear divergence-form equations (see \cite[Proposition~3.3 and Remark~3.4]{DeRosaPigati2025}), we have the following lemma:

\begin{lemma}\label{lem:compactness}
Let $A_{k}(x,p)=D_{p}\Phi_{k}(x,p)$   satisfy  \eqref{e2.3}, \eqref{e2.5} and \eqref{eq:base-A} and suppose $A_k(x,0)=0$.
 Suppose that $v_{k}$ is a weak solution of
\begin{equation}\label{e2.9}
 -\operatorname{div}\bigl(A_{k}(x,\nabla v_{k})\bigr)+W'(v_{k})=0
\end{equation}
 with \(|v_k|\leq1\) a.e. in $B_{2R}\subset\R^{n}$.  Then there exists $\alpha\in(0,1)$  such that
\begin{equation}\label{e2.10}
 \|v_{k}\|_{C^{1,\alpha}(B_{R})}\leq C_{R},
\end{equation}
where $C_{R}$ is a constant independent of $k$.
If $A_{k}\to A_{\infty}$ locally uniformly,  then, up to a subsequence, $v_k\to v_\infty$ in $C^{1,\alpha'}_{\mathrm{loc}}$ for every $0<\alpha'<\alpha$, where $v_{\infty}$ is a weak solution of the following limiting equation
\[ -\operatorname{div}\bigl(A_{\infty}(x,\nabla v_{\infty})\bigr)+W'(v_{\infty})=0.\]
\end{lemma}

\begin{lemma}[Comparison principles]\label{lem:maximum-principles}
Let $B$ be a bounded measurable symmetric matrix with $\lambda I\leq B\leq\Lambda I$, and let $c\in L^{\infty}$. In (i) and (ii), assume
\(w\in H^1_{\mathrm{loc}}(\Omega)\cap C(\Omega)\).
\begin{enumerate}
\item[(i)] If $w\geq0$    and
\[
 -\operatorname{div}(B\nabla w)+cw>0
\]
weakly in a connected open set $\Omega$, then $w>0$ in $\Omega$.
\item[(ii)] If $w\geq0$ satisfies
\[
 -\operatorname{div}(B\nabla w)+cw=0
\]
 in $\Omega$ and vanishes at an interior point, then $w\equiv0$ in $\Omega$.
\item[(iii)] Let \(A_0=D\Phi_0:\mathbb R^n\to\mathbb R^n\)
be autonomous and satisfy \eqref{e2.3} and \eqref{e2.5}.  Let $z_{1},z_{2}$ be bounded entire weak solutions of
\[
 \operatorname{div}A_{0}(\nabla z)-W'(z)=0,
\]
and set $w=z_{1}-z_{2}$.  Suppose that
\[\label{e2.11}
 \frac{W'(z_{1})-W'(z_{2})}{z_{1}-z_{2}}\geq\kappa>0
 \qquad\text{on }\{w<0\},
\]
where the quotient is defined by continuity when the $z_1=z_2$.  Then $w\geq0$ in $\R^{n}$. And the conclusion also
holds in a half-space \(H\), provided that
\eqref{e2.11} holds on \(H\cap\{w<0\}\) and
\(w\geq0\) on \(\partial H\).
\end{enumerate}
\end{lemma}

\begin{proof}
Choose \(C>\|c^-\|_{L^\infty}\). For \(w\ge0\),
\[-\operatorname{div}(B\nabla w)+cw>0
\;\Longrightarrow\;
-\operatorname{div}(B\nabla w)+(c+C)w>0.
\]
Since \(c+C\ge0\),  both (i) and (ii) follow from the weak Harnack inequality and the strong minimum principle.

Now we show (iii). By a direct calculation, we get
\[ \operatorname{div}\left(A_{0}(\nabla z_1)-A_{0}(\nabla z_2)\right)-\left(W'(z_1)-W'(z_2)\right)=0.\]
Using \eqref{e2.5}, there is a symmetric measurable matrix \(B_{12}\) with \(\lambda I\le B_{12}\le \Lambda I\) such that \(A_0(\nabla z_1)-A_0(\nabla z_2)=B_{12}\nabla w\). Consequently,
\[ \operatorname{div}\left(B_{12}\nabla w\right)-\frac{W'(z_1)-W'(z_2)}{z_1-z_2}w=0.\]
Suppose that $i=\inf_{\R^{n}}w<0$, and choose $x_{k}$ with $w(x_{k})\to i$.  Translate both solutions by $x_{k}$ and set $z_1^k(x):=z_1(x+x_k)$, $z_2^k(x):=z_2(x+x_k)$.  By lemma \ref{lem:compactness}, up to a subsequence  still denote by $z_1^k$ and $z_2^k$  that $z_1^k\to\overline{z}_1$ and $z_2^k\to\overline{z}_2$. Let $w_{\infty}(x)=\overline{z}_1-\overline{z}_2$, then we have
\[w_{\infty}(0)=i\quad\text{ and }\quad w_{\infty}(x)\geq i \text{ for all }x\in\mathbb{R}^n.\]
Because \(i<0\), the continuity of \(w_\infty\) gives a neighborhood \(U\ni0\) on which \(w_\infty<0\) and thus
\[
c_\infty:=\frac{W'(\bar z_1)-W'(\bar z_2)}{\bar z_1-\bar z_2}\ge \kappa>0\quad\text{in }U.
\]
Applying \eqref{e2.5} to
\(\nabla\bar z_1\) and \(\nabla\bar z_2\),
we obtain a symmetric measurable matrix \(B_\infty\)
with \(\lambda I\leq B_\infty\leq\Lambda I\) such that
\[
 A_0(\nabla\bar z_1)-A_0(\nabla\bar z_2)
 =B_\infty\nabla w_\infty.
\]
Hence
\[
 -\operatorname{div}(B_\infty\nabla w_\infty)
 +c_\infty w_\infty=0
 \qquad\text{in }U.
\]
Set $z:=w_\infty-i$, then we get
\[
-\operatorname{div}(B_\infty\nabla z)+c_\infty z=-i\,c_\infty>0
\]
in \(U\). Since \(i<0\) and \(c_\infty\ge\kappa>0\). Thus \(z\ge0\) is a strict supersolution that vanishes at the interior point \(0\), contradicting (i).
Therefore \(i\ge0\), i.e. $w=z_1-z_2\ge0$ in $\mathbb R^n$.

By a similar argument, the same conclusion holds in the half-space under the boundary condition $w\geq0$, which prevents minimizing sequences from approaching the boundary.

\end{proof}

Now we claim a nondegeneracy away from the nodal set.

\begin{lemma}\label{lem:nondegeneracy}
There exist constants \(R_0>1\), \(\eta_0\in(0,1)\), and \(\varepsilon_0>0\), uniform in \(p\in M\), such that if \(|u|\leq1\) is a weak solution of \eqref{e1.3} on a neighborhood of $\overline{B_{2R_0\varepsilon}(p)}\subset\subset M$ with \(0<\varepsilon<\varepsilon_0\), and
\[
B_{2R_0\varepsilon}(p)\cap\{u=0\}=\varnothing,
\]
  then \(|u(p)|\ge \eta_0\).
\end{lemma}
\begin{proof}
We may assume that $u>0$ in the ball. By means of the exponential map at \(p\), define $\varphi_\epsilon(y):=\exp_p(\epsilon y)$ for any $y\in B_{2R_0}(0)\subset T_pM$ and set $v(y):=u(\varphi_\epsilon(y))$. Then
\[B_{2R_0}\cap\{v=0\}=\emptyset,\]
and $v(y)$ satisfies
\[
 -\operatorname{div}_{g_{\varepsilon,p}}
 \left(A(\varepsilon y,\nabla_{g_{\varepsilon,p}}v)\right)
 +W'(v)=0.
\]
where $g_{\varepsilon,p}(y)=(\exp^*_p g)(\varepsilon y)$. Since $g_{\varepsilon,p}(y)_{ij} = \delta_{ij} + \frac{1}{3}R_{ikjl}(p)(\varepsilon y^k)(\varepsilon y^l) + O(\varepsilon^3|y|^3)$,  we have
\[
 |g_{\varepsilon,p}(y)_{ij}-\delta_{ij}|
 \le C R_0^2\varepsilon^2,
\]
where \(C\) is independent of \(p\).
Similarly,
\[
 \partial_k g_{\varepsilon,p}(y)_{ij}
 =
 \varepsilon\,(\partial_k g_{ij})(\varepsilon y)
 =
 O(\varepsilon^2|y|),
\]
and hence
\[
 |\partial_k g_{\varepsilon,p}(y)_{ij}|
 \le C R_0\varepsilon^2\to0.
\]
Consequently, \(g_{\varepsilon,p}\to I\) in
\(C^1(B_{R_0})\), uniformly in \(p\), by compactness of \(M\).

Let $\zeta\in C^{\infty}(\mathbb{R}^n)$ such that
\[0\leq\zeta\leq1,\quad \zeta(y)\equiv1\text{ in }B_{R_0/2},\quad \zeta(y)\equiv0\text{ in }B_{R_0}^c\quad\text{and}\quad |\nabla\zeta|\leq C_0/R_0.\]
If we choose $R_0$ large enough such that $R_0> C_0\sqrt{\Lambda2^n}$, then we have
\begin{equation}\label{e3.3}
 \Lambda\int_{B_{R_{0}}}|\nabla\zeta|^{2}dy<\int_{B_{R_{0}}}\zeta^{2}dy.
\end{equation}
Set
\begin{equation}\label{e3.4}
 \mathcal{J}(z)=\int_{B_{R_{0}}}\left(\Phi( \varepsilon y,\nabla_{g_{\varepsilon,p}} z)+W(z)-W(0)\right)\,\mathrm{d}V_{g_{\varepsilon,p}}
\end{equation}
then we get
\[
\begin{split}
 \mathcal J(t\zeta)
 ={}&
 t^2\int_{B_{R_0}}
 \Phi(\varepsilon y,\nabla_{g_{\varepsilon,p}}\zeta)
 \,\mathrm dV_{g_{\varepsilon,p}}
 -\frac{t^2}{2}\int_{B_{R_0}}\zeta^2
 \,\mathrm dV_{g_{\varepsilon,p}}
 +\frac{t^4}{4}\int_{B_{R_0}}\zeta^4
 \,\mathrm dV_{g_{\varepsilon,p}}\\
 \le{}&
 \frac{t^2\Lambda}{2}\int_{B_{R_0}}
 |\nabla_{g_{\varepsilon,p}}\zeta|_{g_{\varepsilon,p}}^2
 \,\mathrm dV_{g_{\varepsilon,p}}
 -\frac{t^2}{2}\int_{B_{R_0}}\zeta^2
 \,\mathrm dV_{g_{\varepsilon,p}}
 +\frac{t^4}{4}\int_{B_{R_0}}\zeta^4
 \,\mathrm dV_{g_{\varepsilon,p}}\\
 \le{}&
 \frac{t^2}{2}
 \left(
 \Lambda\int_{B_{R_0}}|\nabla\zeta|^2\,dy
 -\int_{B_{R_0}}\zeta^2\,dy
 +C\varepsilon^2R_0^{n+2}
 \right)
 +\frac{t^4}{4}\int_{B_{R_0}}\zeta^4
 \,\mathrm dV_{g_{\varepsilon,p}}.
\end{split}
\]
Since the inequality in \eqref{e3.3} is strict,
we may choose \(\varepsilon_0>0\) sufficiently small
and then fix \(t_0\in(0,1)\) sufficiently small such that
\[
 \mathcal J(t_0\zeta)\le-a_0<0
\]
for some \(a_0>0\), uniformly in \(p\in M\) and
\(0<\varepsilon<\varepsilon_0\).

By the direct method, \(\mathcal J\) has a minimizer
\(z\in H^1_0(B_{R_0})\cap L^4(B_{R_0})\).
Since
\[
 \mathcal J(z)\le\mathcal J(t_0\zeta)\le-a_0,
\]
the minimizer is nonzero.
Moreover, \(\mathcal J(|z|)=\mathcal J(z)\), and
truncation at \(1\) does not increase the energy.
Thus we may choose \(0\le z\le1\).

The minimizer satisfies the Euler-Lagrange equation.
By interior regularity and boundary continuity for
the zero Dirichlet problem,
\[
 z\in C^{1,\alpha}_{\mathrm{loc}}(B_{R_0})
       \cap C(\overline{B_{R_0}}),
 \qquad z=0\text{ on }\partial B_{R_0}.
\]
The strong minimum principle then gives
\(z>0\) in \(B_{R_0}\).


We show that \(z(0)\ge\eta_0\) for some constant
\(\eta_0\in(0,1)\), uniformly in \(p,\varepsilon\).
Otherwise, there exist \(p_j\in M\),
\(0<\varepsilon_j<\varepsilon_0\), and corresponding
minimizers \(z_j\) such that \(z_j(0)\to0\).
Let
\[
 g_j=g_{\varepsilon_j,p_j},
\]
and let \(\mathcal J_j\) denote the corresponding functional.
By the local estimates used in
Lemma~\ref{lem:compactness}, a further subsequence satisfies
\[
 z_j\longrightarrow z_\infty
 \quad\text{in }C^{1,\alpha'}_{\mathrm{loc}}(B_{R_0}),
 \qquad 0<\alpha'<\alpha,
\]
where \(z_\infty\) is a nonnegative solution of
the limiting equation and \(z_\infty(0)=0\).
The strong minimum principle gives
\(z_\infty\equiv0\) in \(B_{R_0}\).
For any \(\sigma>0\), choose \(\delta>0\) sufficiently
small that
\[
 \operatorname{Vol}_{g_j}(B_{R_0}^{\delta})<\sigma
 \qquad\text{for all }j,
\]
where
$
 B_{R_0}^{\delta}
 :=
 \left\{
 y\in B_{R_0}:
 \operatorname{dist}(y,\partial B_{R_0})<\delta
 \right\}.
$
Since \(z_j\to0\) locally uniformly and \(0\le z_j\le1\),
we obtain
\[
 \int_{B_{R_0}}z_j^2\,\mathrm dV_{g_j}
 ={}
 \int_{B_{R_0}^{\delta}}z_j^2\,\mathrm dV_{g_j}
 +
 \int_{B_{R_0}\setminus B_{R_0}^{\delta}}
 z_j^2\,\mathrm dV_{g_j}
 \le{}2\sigma
\]
for all sufficiently large \(j\).  Then, the arbitrariness of \(\sigma>0\) implies that $ \int_{B_{R_0}}z_j^2\,\mathrm dV_{g_j}\longrightarrow0.$\\
On the other hand, since  $
 W(z_j)-W(0)
 =-\frac12z_j^2+\frac14z_j^4
 \ge-\frac12z_j^2, $ we have
\[
 \mathcal J_j(z_j)
 \ge-\frac12\int_{B_{R_0}}z_j^2\,\mathrm dV_{g_j}.
\]
Hence $ \liminf_{j\to\infty}\mathcal J_j(z_j)\ge0,$, which is a contradiction with $
 \mathcal J_j(z_j)
 \le\mathcal J_j(t_0\zeta)
 \le-a_0<0.
$
Therefore \(z(0)\ge\eta_0\).


For \(0<\theta<1\), the one-homogeneity of \(A\)
and \(W'(s)=s^3-s\) give
\begin{equation}\label{e3.5}
 \operatorname{div}_{g_{\varepsilon,p}}
 \left(
 A(\varepsilon y,
   \nabla_{g_{\varepsilon,p}}(\theta z))
 \right)
 -W'(\theta z)
 =
 \theta(1-\theta^2)z^3>0
 \qquad\text{in }B_{R_0}.
\end{equation}

Rescaling back, set
\[
 z_\varepsilon(x)
 :=
 z(\varphi_\varepsilon^{-1}(x)),
 \qquad x\in B_{R_0\varepsilon}(p).
\]
Then \(z_\varepsilon\) solves \eqref{e1.3} in
\(B_{R_0\varepsilon}(p)\), vanishes on its boundary,
and satisfies
\[
 z_\varepsilon(p)=z(0)\ge\eta_0.
\]
Since \(u>0\) in \(B_{2R_0\varepsilon}(p)\) and
$
 \overline{B_{R_0\varepsilon}(p)}
 \subset B_{2R_0\varepsilon}(p),
$
we have
$
 m:=\min_{\overline{B_{R_0\varepsilon}}(p)}u>0.
$
 Define
\[S=\{\theta\in(0,1): u-\tau z_{\varepsilon}\geq0\text{ for all }0\leq\tau\leq\theta.\}\]
Obviously, $\theta z<m\leq u$ for $0<\theta<
 \min\left\{
 1,\frac{m}
 {\|z_\varepsilon\|_{L^\infty(B_{R_0\varepsilon}(p))}}
 \right\}$ and thus $S\neq\emptyset$. Set $\theta_0:=\sup S$, we show that $\theta_0=1$. Assume by contradiction that $\theta_0<1$ and set $w:=u-\theta_0 z_\varepsilon$, then $w\geq0$ and
 \[
 -\varepsilon^2\operatorname{div}_g(B\nabla w)+cw
 =
 \theta_0(1-\theta_0^2)z_\varepsilon^3>0
 \qquad\text{in }B_{R_0\varepsilon}(p),
\]
where \(\lambda I\le B\le\Lambda I\) and
$
c=\frac{W'(u)-W'(\theta_0 z_\varepsilon)}{u-\theta_0 z_\varepsilon},
$
with the quotient defined by continuity when \(w=0\).
Since \(z_\varepsilon=0\) on
\(\partial B_{R_0\varepsilon}(p)\), we have
\(w>0\) on \(\partial B_{R_0\varepsilon}(p)\).

If \(w>0\) everywhere on
\(\overline{B_{R_0\varepsilon}(p)}\), then continuity
and compactness give
$
 \delta
 :=
 \min_{\overline{B_{R_0\varepsilon}}(p)}w>0.
$
Set
$
 h=
 \min\left\{
 \frac{\delta}{2},
 \frac{1-\theta_0}{2}
 \right\}>0.
$
Since \(0\le z_\varepsilon\le1\), we have
\[
 u-(\theta_0+h)z_\varepsilon
 =
 w-hz_\varepsilon
 \ge\delta-h
 \ge\frac{\delta}{2}>0.
\]
Thus \(\theta_0+h\in S\), contradicting
the definition of \(\theta_0\).

 Therefore, $w(x_0)=0$ for some $x_0\in B_{R_0\varepsilon}(p)$, which is a contradiction with lemma \ref{lem:maximum-principles} (i).  Hence $u\geq z_\varepsilon$ and $u(p)\geq z_\varepsilon(p)\geq\eta_{0}$.  The negative case follows by replacing $u$ by $-u$.
\end{proof}

\begin{lemma}\label{lem:pure-decay}
Let $u$ be a solution of \eqref{e1.3} with $|u|\leq 1$ and \(0<\varepsilon<\varepsilon_0\). Let $\Omega$ be a connected component of $M\setminus\{u=0\}$  and let $s\in\{-1,1\}$ be the sign of $u$ on $\Omega$.  If $B_{2R_{0}\varepsilon}(x)\subset\Omega$, then there are $C>0$ and $\gamma>0$ such that
\begin{equation}\label{e3.6}
 0\leq1-su(x)\leq C e^{-\frac{\gamma d}{\varepsilon}},
\end{equation}
where $d=\operatorname{dist}(x,\{u=0\})$.
Moreover,   fix \(0<k_3<k_1<k_2\) and let $\Omega_{\varepsilon}:=\{x\in \Omega: \ k_1\varepsilon<d<k_2\varepsilon\}$, then
\begin{equation}\label{e3.7}
 \int_{\Omega_{\varepsilon}}\left(\varepsilon F(x,\nabla u)^{2}+\frac{W(u)}{\varepsilon}\right)\,\mathrm{d}V_{g}
 \leq C\varepsilon^{-1}\int_{\Omega_{\varepsilon}^{*}}(1-su)^{2}\,\mathrm{d}V_{g},
\end{equation}
where $\Omega_{\varepsilon}^{*}:=\{x\in \Omega: \operatorname{dist}(x, \Omega_{\varepsilon})<k_3\varepsilon\}$ with $k_3<k_1$.
\end{lemma}
\begin{proof}
Set $w=1-su$, then we have
\begin{equation}\label{e3.8}
 -\varepsilon^{2}\operatorname{div}_{g}A(x,\nabla w)+c(x)w=0
 \end{equation}
where $c(x)=|u(x)|(1+|u(x)|)$. Since $B_{2R_{0}\varepsilon}(x)\subset\Omega$, it follows from lemma \ref{lem:nondegeneracy} that $c(x)\geq\eta_0(1+\eta_0)>0$. Moreover, by \eqref{e2.5}, there exists a  measurable symmetric matrix $B(x)$ with $\lambda I\leq B(x)\leq\Lambda I$ such that
\[A(x,\nabla w)=A(x,\nabla w)-A(x,0)=B(x)\nabla w.\]
Therefore \eqref{e3.8} can be rewritten as
\begin{equation}\label{eq3.8} -\varepsilon^{2}\operatorname{div}_{g}\left(B(x)\nabla w\right)+c(x)w=0.\end{equation}
If \(d=+\infty\), then \(\Omega\) is a connected
component of \(M\). Integrating \eqref{eq3.8} over
\(\Omega\) gives
$
 \int_\Omega cw\,\mathrm dV_g=0.
$
Since \(c>0\) and \(w\ge0\), we obtain \(w\equiv0\),
and the conclusion follows.

If \(d/\varepsilon\le2R_0+2\), estimate
\eqref{e3.6} follows from \(0\le w\le1\)
after enlarging \(C\).
We may therefore assume that
\[
 2R_0+2<d/\varepsilon<+\infty.
\]

Let $B_{r\varepsilon}(x)$ be a ball centered at $x$ with $r\varepsilon$ such that $r\leq d/\varepsilon-2R_0$, where $d=\operatorname{dist}(x,\{u=0\})$. Then, for every $y\in B_{r\varepsilon}(x)$, we have $B_{2R_0\varepsilon}(y)\cap\{u=0\}=\emptyset$, hence $c(y)\geq\eta_0(1+\eta_0)>0$. Let \(2\le r\le d/\varepsilon-2R_0\).
Choose \(\eta\in C_c^\infty(B_{r\varepsilon}(x))\)
such that
\[
 0\le\eta\le1,\qquad
 \eta\equiv1\text{ on }B_{(r-1)\varepsilon}(x),
 \qquad
 \eta\equiv0\text{ outside }B_{r\varepsilon}(x),
 \qquad
 |\nabla\eta|\le\frac{C}{\varepsilon}.
\]
We choose the test function $\eta^2w$ in \eqref{eq3.8} and  integrate by parts to obtain
\[\varepsilon^2\int_{B_{r\varepsilon}(x)}\eta^2 B\nabla w\cdot\nabla w \ dV_g+2\varepsilon^2\int_{B_{r\varepsilon}(x)}\eta w B\nabla w\cdot\nabla\eta \ dV_g+\int_{B_{r\varepsilon}(x)}c\eta^2w^2\  dV_g=0.\]
Since $\lambda I\leq B\leq\Lambda I$ and $c\geq\eta_0(1-\eta_0):=c_0$, using Young's inequality, we obtain
\[
\begin{split}
 &\varepsilon^2\lambda
 \int_{B_{r\varepsilon}(x)}
 \eta^2|\nabla w|^2\,\mathrm dV_g
 +c_0\int_{B_{r\varepsilon}(x)}
 \eta^2w^2\,\mathrm dV_g\\
 &\qquad\le
 \frac{\varepsilon^2\lambda}{2}
 \int_{B_{r\varepsilon}(x)}
 \eta^2|\nabla w|^2\,\mathrm dV_g
 +\frac{2\Lambda^2}{\lambda}\varepsilon^2
 \int_{B_{r\varepsilon}(x)}
 |\nabla\eta|^2w^2\,\mathrm dV_g.
\end{split}
\]
Since $\eta\equiv1$ in $B_{(r-1)\varepsilon}(x)$ and $|\nabla\eta|\leq C/\varepsilon$, we obtain
\begin{equation}\label{e3.9}\varepsilon^2\int_{B_{(r-1)\varepsilon}(x)} |\nabla w|^2\ dV_g+c_0\int_{B_{(r-1)\varepsilon}(x)}w^2\ dV_g\leq C\int_{B_{r\varepsilon}(x)\setminus B_{(r-1)\varepsilon}(x)} w^2 \ dV_g, \end{equation}
where $C$ is a constant depend only on $\lambda,\Lambda$. Set $I(r)=\int_{B_{r\varepsilon}(x)}w^2\,\mathrm dV_g$,  then \eqref{e3.9} implies
\[c_0I(r-1)\leq C_1 \left(I(r)-I(r-1)\right).\]
Therefore,
\[I(r-1)\leq\theta I(r),\]
where $\theta=\frac{C_1}{c_0+C_1}\in(0,1)$.
Set
\[
 r_m=\frac d\varepsilon-2R_0,
 \qquad
 a=-\log\theta>0,
 \qquad
 N=\lfloor r_m-1\rfloor.
\]
Iterating the preceding inequality gives
\[
 I(1)
 \le I(r_m-N)
 \le\theta^N I(r_m)
 \le C e^{-a r_m}I(r_m),
\]
where we used \(1\le r_m-N<2\).
Since \(0\le w\le1\) and \(r_m\varepsilon\le d\), we have
\[
 I(r_m)
 \le\operatorname{Vol}_g(B_d(x))
 \le C d^n.
\]
Consequently,
\[
 I(1)
 \le C d^n
 e^{
 -a\left(\frac d\varepsilon-2R_0\right)
 }.
\]
Using the local boundedness estimate in
\(B_\varepsilon(x)\), we obtain
\[
\begin{split}
 w(x)
 \le C\varepsilon^{-n/2}I(1)^{1/2}
 \le
 C\left(\frac d\varepsilon\right)^{n/2}
 e^{
 -\frac a2\left(\frac d\varepsilon-2R_0\right)
  }
 \le C e^{-\gamma d/\varepsilon},
\end{split}
\]
where we may take \(\gamma=a/4>0\).
This proves \eqref{e3.6}.


For the energy estimate, choose
\(\eta\in C_c^\infty(\Omega_\varepsilon^*)\) such that
\[
 0\le\eta\le1,\qquad
 \eta\equiv1\text{ on }\Omega_\varepsilon,
 \qquad
 |\nabla\eta|\le\frac{C}{k_3\varepsilon}.
\]
Testing \eqref{eq3.8} with \(\eta^2w\) and arguing
as above, we obtain
\[
 \frac{\lambda\varepsilon^2}{2}
 \int_{\Omega_\varepsilon^*}
 \eta^2|\nabla w|^2\,\mathrm dV_g
 +
 \int_{\Omega_\varepsilon^*}
 c\eta^2w^2\,\mathrm dV_g
 \le
 \frac{2\Lambda^2}{\lambda}\varepsilon^2
 \int_{\Omega_\varepsilon^*}
 |\nabla\eta|^2w^2\,\mathrm dV_g.
\]
Since \(c\ge0\), it follows that
\[
 \varepsilon^2
 \int_{\Omega_\varepsilon}|\nabla w|^2\,\mathrm dV_g
 \le
 C\int_{\Omega_\varepsilon^*}w^2\,\mathrm dV_g,
\]
where \(C\) may depend on \(k_3\).
Therefore,
\[\varepsilon\int_{\Omega_{\varepsilon}}F(x,\nabla u)^2\ dV_g\leq\varepsilon\Lambda\int_{\Omega_{\varepsilon}} |\nabla w|^2\ dV_g\leq C\varepsilon^{-1}\int_{\Omega_{\varepsilon}^*} w^2 \ dV_g.\]
On the other hand, since $ W(u)=\frac{1}{4}(1-su)^{2}(1+s u)^{2}\leq(1-su)^{2}=w^2$, we have
\[ {\varepsilon}^{-1} \int_{\Omega_{\varepsilon}} W(u)\ dV_g \le {\varepsilon}^{-1} \int_{\Omega_{\varepsilon}} w^2\ dV_g \le C\varepsilon^{-1} \int_{\Omega_{\varepsilon}^*} w^2 \ dV_g.\]
Combining the two estimates, we obtain \eqref{e3.7}.

\end{proof}

\section{Strict stability and anisotropic barriers}\label{sec:barriers}

 The first variation of the anisotropic area functional \(\mathcal{A}_F\) is given by
\begin{equation}\label{e4.1}
\left.\frac{\mathrm{d}}{\mathrm{d}s}\right|_{t=0}\mathcal{A}_F(\Sigma_t)
=
\int_\Sigma H_F(\Sigma,\nu)\,\varphi\,\mathrm{d}\mathcal{H}_g^{n-1},
\end{equation}
where \(H_F\) denotes the anisotropic mean curvature, given by
\begin{equation}\label{e4.2}
H_F(\Sigma,\nu)
=
D_xF(x,\nu)\cdot\nu
+
\operatorname{div}_\Sigma D_\xi F(x,\nu).
\end{equation}

\begin{lemma}\label{lem:foliation}
Let \(\Gamma\subset M\) be a smooth, closed, connected, embedded, two-sided hypersurface
which is \(F\)-minimal and strictly \(F\)-stable.
Then there exist \(s_0>0\), \(c_0>0\), and a positive smooth function \(\phi\in C_c^\infty(\Gamma)\)
such that the normal graphs
\begin{equation}\label{e4.3}
\Gamma_s=\left\{\exp_x(s\phi(x)\nu_\Gamma(x)):x\in\Gamma\right\},\qquad |s|<s_0,
\end{equation}
form a foliation of a neighborhood of \(\Gamma\) and satisfy
\begin{equation}\label{e4.4}
H_F(\Gamma_s,\nu_s)\ge c_0s\qquad\text{for }0<s<s_0.
\end{equation}
where \(\nu_s\) is chosen smoothly in \(s\) with
\(\nu_0=\nu_\Gamma\). Moreover, if the orientation of $\Gamma$ is reversed (i.e., $\nu_\Gamma$ is replaced by $-\nu_\Gamma$), then an analogous foliation exists on the other side of $\Gamma$, and the corresponding estimate for $H_F$ holds as well.

\end{lemma}
\begin{proof}
For \(f\in C^{2,\alpha}(\Gamma)\) with sufficiently small norm, set
\[
\Psi_f(x):=\exp_x\bigl(f(x)\nu_\Gamma(x)\bigr),
\qquad
\Gamma(f):=\Psi_f(\Gamma),
\]
and let \(\nu_f\) be the unit normal to \(\Gamma(f)\) chosen
continuously in \(f\) with \(\nu_0=\nu_\Gamma\).
Define
\(\mathscr H:\mathcal U\to C^{0,\alpha}(\Gamma)\),
where \(\mathcal U\) is a sufficiently small neighborhood of
\(0\) in \(C^{2,\alpha}(\Gamma)\), by
\[
\mathscr H(f)(x)
:=H_F(\Gamma(f),\nu_f)\bigl(\Psi_f(x)\bigr).
\]
Since $\Gamma$ is $F$-minimal, we have
\[
\mathscr{H}(0)=H_F(\Gamma,\nu_\Gamma)=0.
\]
For \(\varphi,\psi\in C^{2,\alpha}(\Gamma)\), the second variation of \(\mathcal{A}_F\) at \(\Gamma\) is given by
\[
\delta^2\mathcal{A}_F(\Gamma)[\varphi,\psi]
=
\int_\Gamma \psi L_F\varphi\,\mathrm d\mathcal H_g^{n-1},
\]
where \(L_F:=D\mathscr{H}(0)\) is self-adjoint and uniformly elliptic. 
Moreover, the strict \(F\)-stability of $\Gamma$ implies that $\delta^2\mathcal{A}_F(\Gamma)[\varphi,\varphi]>0$ for any $\varphi\neq0$. Hence \(L_F\) is positive definite, its first eigenvalue
\[
\lambda_1
=
\inf_{\varphi\not\equiv0}
\frac{\int_\Gamma \varphi L_F \varphi \,\mathrm{d}\mathcal{H}_g^{n-1}}
{\int_\Gamma \varphi^2 \,\mathrm{d}\mathcal{H}_g^{n-1}}
>0,
\]
and there exists a positive principal eigenfunction \(\phi>0\) satisfying
\(L_F\phi=\lambda_1\phi\).

\medskip
Now, we define the family of hypersurfaces
\[
\Gamma_s:=\bigl\{\exp_x(s\phi(x)\nu_\Gamma(x)):x\in\Gamma\bigr\},\qquad |s|<s_0,
\]
where \(\exp_x\) denotes the exponential map at \(x\).
Since
\[
\mathscr H(s\phi)
=
sL_F\phi+O(s^2)
=
s\lambda_1\phi+O(s^2)
\qquad\text{in }C^0(\Gamma),
\]
 there exist \(s_0>0\) sufficiently small and \(c_0>0\) such that
 \[
H_F(\Gamma_s,\nu_s)
\ge
\frac{\lambda_1m_0}{2}s
\qquad
\text{for }0<s<s_0,
\]
 where $m_0:=\min_\Gamma\phi>0$. Moreover, since \(\phi>0\) and the exponential map is a local diffeomorphism in the normal direction,
the hypersurfaces \(\{\Gamma_s\}_{|s|<s_0}\) form a foliation of a tubular neighborhood of \(\Gamma\).

\medskip
Finally, from the same expansion, after decreasing \(s_0\) if necessary,
gives
\[
H_F(\Gamma_{-s},\nu_{-s})\leq-c_0s
\qquad\text{for }0<s<s_0.
\]
Since \(F\) is even, reversing the normal changes the sign of \(H_F\).
Consequently,
\[
H_F(\Gamma_{-s},-\nu_{-s})\geq c_0s
\qquad\text{for }0<s<s_0,
\]
which gives the required estimate on the other side of \(\Gamma\).

\end{proof}

Building on the foliation constructed in lemma \ref{lem:foliation}, we now construct the anisotropic signed distance function $d_s$.

\begin{lemma}\label{lem:distance}
Let \(\Gamma\) and \(s_0\) be as in Lemma~\ref{lem:foliation}.
After decreasing \(s_0\) if necessary, there exist a fixed
neighborhood \(\mathcal N\) of \(\Gamma\), containing every
\(\Gamma_s\) with \(|s|<s_0\), and functions
\(d_s\in C^3(\mathcal N)\) such that
\begin{equation}\label{e4.5}
d_s=0\quad\text{on }\Gamma_s,
\qquad
F(x,\nabla d_s)=1\quad\text{in }\mathcal N.
\end{equation}
The map \((s,x)\mapsto d_s(x)\) is \(C^3\), and
$
\nabla d_s=\frac{\nu_s}{F(x,\nu_s)}
\quad\text{on }\Gamma_s.
$
On every level set \(\{d_s=t\}\), we have
\begin{equation}\label{e4.6}
\operatorname{div}_g A(x,\nabla d_s)
=
H_F\left(\{d_s=t\},
\frac{\nabla d_s}{|\nabla d_s|_g}\right).
\end{equation}
Moreover, there exists \(C>0\), independent of \(s\), such that
\begin{equation}\label{e4.7}
\operatorname{div}_g A(x,\nabla d_s)
\geq c_0s-C|d_s|
 \quad\text{in} \quad \mathcal N
\end{equation}
for \(0<s<s_0\).
\end{lemma}

\begin{proof}
Let \(\Gamma_s\) and \(\nu_s\) be the foliation and the unit normal from Lemma~\ref{lem:foliation}.
Fix \(x\in M\). By the Riesz representation theorem, the metric \(g_x\) induces an isometric isomorphism
\[
\mathcal R_x:T_x^*M\longrightarrow T_xM,\qquad
\langle \mathcal R_x p,v\rangle_{g_x}=p(v)
\quad\text{for all }p\in T_x^*M,\ v\in T_xM.
\]
We denote \(p^\sharp:=\mathcal R_x p\) and \(v^\flat:=\mathcal R_x^{-1}v\).
Define the Hamiltonian
\[
\widehat\Phi(x,p):=\Phi(x,\mathcal R_x p),\qquad (x,p)\in T^*M\setminus \{0\}.
\]
On \(\Gamma_s\), set
\[
p_s(x):=\frac{\nu_s^\flat(x)}{F(x,\nu_s(x))}\in T_x^*M.
\]
Since \(F\) is positively homogeneous of degree one,
\[
\widehat\Phi(x,p_s(x))
=
\Phi\left(x,\frac{\nu_s(x)}{F(x,\nu_s(x))}\right)
=
\frac12.
\]
In local coordinates, let \((X(t),P(t))\) solve the Hamilton equations
\begin{equation}\label{e0.1}
\dot X(t)= D_p\widehat\Phi(X(t),P(t)),\qquad
\dot P=-D_x\widehat\Phi(X(t),P(t)),
\end{equation}
with initial data \((X(0),P(0))=(x,p_s(x))\). Then, by the homogeneity of \(\Phi\), we obtain
\[
\dot X(0)
=
A(x,p_s(x)^\sharp)
=
D_\xi\Phi\left(x,\frac{\nu_s(x)}{F(x,\nu_s(x))}\right)
=
D_\xi F(x,\nu_s(x)).
\]
Therefore, Euler's identity yields
\[
\langle \dot X(0),\nu_s\rangle_{g_x}
=
D_\xi F(x,\nu_s)\cdot\nu_s
=
F(x,\nu_s)>0.
\]
We denote by \(X(t;x)\)  the position at time \(t\) of the Hamiltonian flow with initial point \(x\in\Gamma_s\).
Thus the map \(\Theta_s\) from \(\Gamma_s\times(-r,r)\) to a neighborhood of \(\Gamma_s\),
denoted by \(\mathcal N_s\), defined as
\[
\Theta_s(x,t)=X(t;x),
\]
which, after decreasing \(s_0\) if necessary, is a \(C^3\) diffeomorphism onto its image for \(|s|<s_0\).

We now verify that $P$ is precisely the differential of the coordinate $t$.
Let \(v\in T_x\Gamma_s\), and choose a smooth curve
\(x_\tau\in\Gamma_s\) such that
\[
x_0=x,\qquad
\left.\frac{\mathrm d}{\mathrm d\tau}\right|_{\tau=0}x_\tau=v.
\]
Let \((X_\tau,P_\tau)\) solve the Hamilton equations with
initial data
\[
(X_\tau(0),P_\tau(0))
=
(x_\tau,p_s(x_\tau)).
\]
Differentiating   \eqref{e0.1} with respect to \(\tau\) and evaluating at \(\tau=0\), we obtain the linearized Hamilton equations
\[
\begin{aligned}
\dot{\delta X}
&=
D^2_{px}\widehat\Phi(X_0,P_0)\,\delta X
+
D^2_{pp}\widehat\Phi(X_0,P_0)\,\delta P,\\
\dot{\delta P}
&=
-D^2_{xx}\widehat\Phi(X_0,P_0)\,\delta X
-
D^2_{xp}\widehat\Phi(X_0,P_0)\,\delta P,
\end{aligned}
\]
where $(\delta X,\delta P):=\left.\frac{\partial}{\partial\tau}\right|_{\tau=0}(X_\tau,P_\tau).$
Moreover, since $\frac{\mathrm d}{\mathrm dt}\widehat{\Phi}(X_\tau,P_\tau)=D_x\widehat{\Phi}\cdot D_p\widehat{\Phi}+D_p\widehat{\Phi}\cdot(-D_x\widehat{\Phi})=0$, we deduce that \[\widehat{\Phi}(X_\tau(t),P_\tau(t))=\widehat{\Phi}(x_\tau,p_s(x_\tau))=\frac{1}{2},\]
which implies that $\delta \widehat{\Phi}=0$.
Using the linearized Hamilton equations, we obtain
\begin{align*}
 \frac{\mathrm{d}}{\mathrm{d}t}(P\cdot\delta X)
 &=
 -D_{x}\widehat{\Phi}\cdot\delta X
 +P\cdot\bigl(D^{2}_{px}\widehat{\Phi}\,\delta X
 +D^{2}_{pp}\widehat{\Phi}\,\delta P\bigr)\\
 &=
 D_{x}\widehat{\Phi}\cdot\delta X
 +D_{p}\widehat{\Phi}\cdot\delta P
 =
 \delta\widehat{\Phi}
 =
 0,
\end{align*}
where the second equality follows by differentiating
Euler's identity
\(P\cdot D_p\widehat\Phi=2\widehat\Phi\)
with respect to \(x\) and \(p\).  Therefore, $P\cdot\delta X=P(0)\cdot\delta X(0)=p_s(v)=0$ for all $t$. On the other hand,
\[
 P\cdot\dot X=P\cdot D_{p}\widehat{\Phi}=2\widehat{\Phi}=1.
\]
Thus, in the coordinates induced by \(\Theta_s\),   we conclude that $P=dt$.

Define the anisotropic signed distance function \(d_s:\mathcal N_s\to\mathbb R\) by
\[
d_s(\Theta_s(x,t))=t.
\]
Then we have
\[
\mathrm{d}d_s=\mathrm{d}t=P,
\qquad
\nabla d_s=(\mathrm{d}d_s)^\sharp=P^\sharp.
\]
Since \(d_s(\Theta_s(x,0))=0\) and
\[
\widehat{\Phi}(X(t),P(t))
=
\Phi(X(t),P(t)^\sharp)
=
\frac12 F(X(t),P(t)^\sharp)^2
=
\frac12,
\]
it follows that \(F(x,\nabla d_s)=1\) on \(\mathcal N_s\).
After decreasing \(s_0\) further, we may choose a fixed
neighborhood \(\mathcal N\) of \(\Gamma\) such that, for
every \(|s|\leq s_0\),
\[
\Gamma_s\subset\mathcal N,
\qquad
\overline{\mathcal N}
\subset
\Theta_s\bigl(\Gamma_s\times(-r/2,r/2)\bigr).
\]
Restricting \(d_s\) to \(\mathcal N\), we obtain \eqref{e4.5}. Moreover, the initial data give
\[
\nabla d_s=\frac{\nu_s}{F(x,\nu_s)}
\quad\text{on }\Gamma_s.
\]


Let \(\Sigma:=\{d_s=t\}\) be a level set, and choose an orthonormal frame
\(e_1,\dots,e_{n-1},\nu\) at a point \(x\in\Sigma\), where
\[
\nu=\frac{\nabla d_s}{|\nabla d_s|_g}.
\]
The divergence of \(A(x,\nabla d_s)\) can be decomposed into
\[
\operatorname{div}_g A(x,\nabla d_s)
=
\sum_{i=1}^{n-1}\langle \nabla_{e_i} A(x,\nabla d_s), e_i \rangle
+
\langle \nabla_\nu A(x,\nabla d_s), \nu \rangle.
\]
Since \(F(x,\nabla d_s)=1\) and \(D_\xi F\) is zero-homogeneous, we have
\[
A(x,\nabla d_s)
=
D_\xi F(x,\nabla d_s)
=
D_\xi F(x,\nu)
\]
in a neighborhood of \(\Sigma\). Therefore the tangential part becomes
\[
\sum_{i=1}^{n-1}\langle \nabla_{e_i} A(x,\nabla d_s), e_i \rangle
=
\sum_{i=1}^{n-1}\langle \nabla_{e_i} D_\xi F(x,\nu), e_i \rangle
=
\operatorname{div}_\Sigma D_\xi F(x,\nu).
\]
For the normal part, since  \(\langle D_{\xi\xi}^2F(x,\nu)[\nabla_\nu \nu],\nu \rangle=0\), we have
\[
\langle \nabla_\nu A(x,\nabla d_s), \nu \rangle=\langle \nabla_\nu D_\xi F(x,\nu), \nu \rangle=\langle \nabla_xD_\xi F(x,\nu)[\nu], \nu \rangle,
\]
where \(\nabla_x\) denotes covariant differentiation.
Differentiating Euler's identity
\[
D_\xi F(x,\nu)\cdot\nu=F(x,\nu)
\]
along the direction \(\nu\) gives
\[
\langle \nabla_xD_\xi F(x,\nu)[\nu], \nu \rangle
=
\nabla_xF(x,\nu)\cdot\nu.
\]
Combining the tangential and normal parts, we obtain
\[
\operatorname{div}_g A(x,\nabla d_s)
=
\nabla_xF(x,\nu)\cdot\nu
+
\operatorname{div}_\Sigma D_\xi F(x,\nu)
=
H_F(\Sigma,\nu),
\]
which proves \eqref{e4.6}.
Taking \(t=0\) and using \eqref{e4.4},  we obtain $\operatorname{div}_g A(x,\nabla d_s)\geq c_0 s$ on $\Gamma_s$ for $0<s<s_0$.
Set
\[
Q_s(x):=\operatorname{div}_g A(x,\nabla d_s),
\qquad x\in\mathcal N_s.
\]
By the \(C^3\) regularity of the family and compactness,
\[
\left|
\frac{\partial}{\partial t}
Q_s\bigl(\Theta_s(y,t)\bigr)
\right|
\leq C
\qquad
\text{for }y\in\Gamma_s,\quad |t|\leq r/2,\quad |s|\leq s_0,
\]
where \(C\) is independent of \(s\). For any \(x\in\mathcal N\), write
\(x=\Theta_s(y,d_s(x))\) with \(y\in\Gamma_s\).
Since \(\Theta_s(y,0)=y\), we have
\[
Q_s(x)-Q_s(y)
=
\int_0^{d_s(x)}
\frac{\partial}{\partial\tau}
Q_s\bigl(\Theta_s(y,\tau)\bigr)\,\mathrm d\tau
\geq -C|d_s(x)|.
\]
Since \(Q_s(y)\geq c_0s\) for \(0<s<s_0\), we obtain
\[
\operatorname{div}_g A(x,\nabla d_s)
=
Q_s(x)
\geq Q_s(y)-C|d_s(x)|
\geq c_0s-C|d_s(x)|.
\]
This proves \eqref{e4.7}.


\end{proof}

\begin{lemma}\label{lem:composition}
Let \(f\in C^2(\mathbb R)\). Then, for any $x\in \mathcal N$,
\begin{equation}\label{eq:composition}
 \operatorname{div}_{g}A\bigl(x,\nabla(f(d_{s}))\bigr)
 =
 f''(d_{s})+f'(d_{s})\operatorname{div}_{g}A(x,\nabla d_{s}).
\end{equation}
\end{lemma}

\begin{proof}
Since \(A\) is one-homogeneous  and \(F(x,\nabla d_s)=1\), we get
\[
A\bigl(x,\nabla(f(d_s))\bigr)
=
f'(d_s)A(x,\nabla d_s),
\]
Therefore, by a direct calculation,
\[
\operatorname{div}_g\bigl(f'(d_s)A(x,\nabla d_s)\bigr)
=
f''(d_s)\nabla d_s\cdot A(x,\nabla d_s)
+
f'(d_s)\operatorname{div}_g A(x,\nabla d_s).
\]
Since $A(x,\nabla d_{s})\cdot\nabla d_{s}=F(x,\nabla d_{s})^{2}=1$, we conclude that \eqref{eq:composition} holds.
\end{proof}

Let $U(t)$ be a one-dimensional solution  defined on $t\in(-\infty,\infty)$, satisfying
\begin{equation}\label{e4.8}
 U''=W'(U),\qquad U(0)=0,
 \qquad \lim_{t\to\pm\infty}U(t)=\pm 1.
\end{equation}
Here \(U'>0\), and both \(1-|U(t)|\) and \(U'(t)\)
decay exponentially as \(|t|\to\infty\).

In the following, we construct strict subsolutions for \eqref{e1.3}.

\begin{lemma}\label{lem:barrier}
After decreasing \(s_0\) if necessary, there exist constants
\(K>1\), \(\kappa>0\), \(\varepsilon_0>0\), and a tubular
neighborhood
\[
N_r(\Gamma)
:=
\{x\in M:\operatorname{dist}_g(x,\Gamma)<r\}
\subset\subset\mathcal N
\]
such that \(\Gamma_s\subset N_{r/4}(\Gamma)\) for
\(0\le s\le s_0\), and, for every \(0<\varepsilon<\varepsilon_0\)
and \(K\varepsilon\le s\le s_0\), the function
\begin{equation}\label{e4.10}
v_s(x)
=
U\left(\frac{d_s(x)}{\varepsilon}\right)
-\kappa\varepsilon s
\end{equation}
satisfies
\begin{equation}\label{e4.11}
\varepsilon^2\operatorname{div}_g A(x,\nabla v_s)
-W'(v_s)>0
\qquad\text{in }N_r(\Gamma).
\end{equation}
\end{lemma}

\begin{proof}
Choose \(r>0\) so that \(N_r(\Gamma)\subset\subset\mathcal N\),
and then decrease \(s_0\) so that
\(\Gamma_s\subset N_{r/4}(\Gamma)\) for \(0\leq s\leq s_0\).
Set \(\delta=\kappa\varepsilon s\) and \(r_s=d_s/\varepsilon\).
Then \(v_s=U(r_s)-\delta\). Using \eqref{eq:composition} and \(U''=W'(U)\), we obtain
\begin{equation}\label{e4.12}
\varepsilon^2\operatorname{div}_g A(x,\nabla v_s)-W'(v_s)
=
W'(U)-W'(U-\delta)
+
\varepsilon U'\operatorname{div}_g A(x,\nabla d_s),
\end{equation}
where
\begin{equation}\label{e4.13}
W'(U)-W'(U-\delta)
=
\delta(3U^2-1)-3U\delta^2+\delta^3.
\end{equation}

Choose \(b\in(1/\sqrt{3},1)\).   In the case of \(|U|\le b\), since $U$ is strictly increasing, there exist \(R_b>0\) such that \(|r_s|\le R_b\). Then, it follows from \eqref{e4.7} that
 \[\operatorname{div}_g A(x,\nabla d_s)\ge c_0s-\frac{CR_b}{K}s\geq\frac{c_0}{2}s\]
 by taking $K\geq\frac{2C R_b}{c_0}$. On the other hand,
\[W'(U)-W'(U-\delta)\geq -\delta-3b\delta^2+\delta^3\geq-2\delta\]
 for $\delta\leq\frac{1}{3b}$.
 Therefore, we have
 \[\varepsilon^2\operatorname{div}_g A(x,\nabla v_s)-W'(v_s)\geq -2\delta+\varepsilon m_b\frac{c_0}{2}s=\left(\frac{c_0m_b}{2}-2\kappa\right)\varepsilon s>0\]
 for $\kappa<\frac{c_0m_b}{4}$ with $m_b:=\min_{[-R_b,R_b]}U'>0$.

 In the case of \(|U|\ge b\). Since $|U|\leq1$, we have
 \[\left|-3U\delta^2+\delta^3\right|\leq 3\delta^2+\delta^3\leq\frac{3b^2-1}{2}\delta\]
 for $\delta\leq\frac{3b^2-1}{8}$. Therefore, we deduce that
 \[W'(U)-W'(U-\delta)\geq \frac{3b^2-1}{2}\delta\quad\text{for }\delta\leq\frac{3b^2-1}{8}.\]
 Moreover, since $\sup_{t\in\mathbb{R}}|t|U'(t)<\infty$, using again \eqref{e4.7} we obtain
\[
\varepsilon U'\operatorname{div}_g A(x,\nabla d_s)
\ge -C\varepsilon^2|r_s|U'(r_s)\geq-C_1\varepsilon^2.
\]
Therefore, we obtain
\[\varepsilon^2\operatorname{div}_g A(x,\nabla v_s)-W'(v_s)\geq\left(\frac{3b^2-1}{2}\kappa K-C_1\right)\varepsilon^2>0\]
 for $K>\frac{2C_1}{(3b^2-1)\kappa}$.
 Fix \(0<\kappa<c_0m_b/4\), and choose
\[
K>
\max\left\{
1,\frac{2CR_b}{c_0},
\frac{2C_1}{(3b^2-1)\kappa}
\right\}.
\]
Finally, choose \(\varepsilon_0>0\) so that
\[
K\varepsilon_0<s_0,
\qquad
\kappa\varepsilon_0s_0
\leq
\min\left\{\frac1{3b},\frac{3b^2-1}{8}\right\}.
\]
Then both cases apply for every
\(0<\varepsilon<\varepsilon_0\) and
\(K\varepsilon\leq s\leq s_0\).
This proves \eqref{e4.11}, and hence \(v_s\) is a strict
subsolution of \eqref{e1.3}.

\end{proof}

\begin{proposition}\label{prop:trapping}
Let $\varepsilon_j\to0$ and $u_{j}\in H^{1}(M)\cap L^{4}(M)$ be global weak solutions of \eqref{e1.3} and
 let \(r>0\) be sufficiently small, as in
Lemma~\ref{lem:barrier}, with \(N_r(\Gamma)\subset\subset\Omega\).
 Suppose that    a smooth, closed, connected, embedded, two-sided and strictly $F$-stable $F$-minimal hypersurface $\Gamma$ in $\Omega\subset M$ is a limit interface of the sequence $\{u_j\}$. Then, there is $C>0$ such that
\begin{equation}\label{e4.14}
 \{u_{j}=0\}\cap N_{r/2}(\Gamma)\subset N_{C\varepsilon_{j}}(\Gamma)
\end{equation}
for all sufficiently large $j$.
\end{proposition}
\begin{proof}
Since $\Gamma$ is a limit interface of $\{u_j\}$ in $\Omega\subset M$, there exists a   neighborhood \(N_{2\rho}(\Gamma)\subset \Omega\) such that
\[
\{u_j=0\}\cap N_{2\rho}(\Gamma)
\longrightarrow\Gamma
\]
in the Hausdorff distance. Decreasing \(r\) and then \(s_0\) if necessary, we may assume
that \(N_r(\Gamma)\Subset N_{2\rho}(\Gamma)\) and
\(\Gamma_s\subset N_{r/4}(\Gamma)\) for \(0\leq s\leq s_0\).
By the Hausdorff convergence,  there exists \(J\) such that for all \(j>J\),
\[
\{u_j=0\}\cap N_{2\rho}(\Gamma)\subset N_{r/4}(\Gamma).
\]
 Since
\[
\{x:r/4<\operatorname{dist}(x,\Gamma)<r\}\subset N_r(\Gamma)\subset\subset N_{2\rho}(\Gamma),
\]
it follows that
\[
\{u_j=0\}\cap \{x:r/4<\operatorname{dist}(x,\Gamma)<r\}=\varnothing
\]
for all sufficiently large \(j\). Since \(\Gamma\) is connected, \(u_j\) has constant sign
on each side of the annulus.
We first consider the positive side.
Replacing \(u_j\) by \(-u_j\) if necessary, we may assume
that \(u_j>0\) on the positive annulus.
This replacement does not change the nodal set.

On the negative boundary of \(N_r(\Gamma)\), compactness gives
\(a>0\) such that \(d_s\leq-a\) uniformly for \(0\leq s\leq s_0\).
Hence, for \(K\varepsilon_j\leq s\leq s_0\),
\[
v_s
\leq -1+Ce^{-c/\varepsilon_j}-\kappa\varepsilon_j s
\leq -1+Ce^{-c/\varepsilon_j}-\kappa K\varepsilon_j^2
<-1\leq u_j
\]
for all sufficiently large \(j\).

On the positive boundary, the distance to \(\{u_j=0\}\)
is bounded below by a positive constant independent of \(j\).
By Lemma~\ref{lem:pure-decay},
\[
u_j\geq1-Ce^{-c/\varepsilon_j}.
\]
Since \(U\leq1\), we have
\[
v_s
\leq1-\kappa\varepsilon_j s
\leq1-\kappa K\varepsilon_j^2
<1-Ce^{-c/\varepsilon_j}
\leq u_j
\]
for all sufficiently large \(j\).
Thus \(v_s<u_j\) on \(\partial N_r(\Gamma)\),
uniformly for \(K\varepsilon_j\leq s\leq s_0\).


Fix \(s=s_0\). By Lemma~\ref{lem:foliation}, \(d_{s_0}\) is strictly negative on \(\Gamma\). Set \[d_*:=-\max_\Gamma d_{s_0}>0.\]
Choose \(0<\eta<d_*/2\) sufficiently small that
\(d_{s_0}<-\eta\) on the closed negative side of \(N_r(\Gamma)\).
Hence
\[
U\left(\frac{d_{s_0}}{\varepsilon_j}\right)
\le -1+C e^{-c\eta/\varepsilon_j}.
\]
By \eqref{e4.10},
\[
v_{s_0}
=
U\left(\frac{d_{s_0}}{\varepsilon_j}\right)-\kappa\varepsilon_j s_0
\le
-1+C e^{-c\eta/\varepsilon_j}-\kappa\varepsilon_j s_0
<-1\le u_j
\]
for sufficiently small \(\varepsilon_j\).

\noindent
The set
\(\{d_{s_0}\geq-\eta\}\cap\overline{N_r(\Gamma)}\)
lies on the positive side of \(\Gamma\) and stays at a fixed
positive distance from \(\Gamma\).
By the Hausdorff convergence and
Lemma~\ref{lem:pure-decay}, there exist constants \(C_1,c_1>0\) such that
\[
u_j\geq1-C_1e^{-c_1/\varepsilon_j}
\]
for all sufficiently large \(j\).
Since \(U\leq1\) and \(s_0\geq K\varepsilon_j\), we have
\[
v_{s_0}
\leq1-\kappa\varepsilon_j s_0
\leq1-\kappa K\varepsilon_j^2
<1-C_1e^{-c_1/\varepsilon_j}
\leq u_j
\]
for all sufficiently large \(j\).
Therefore,
\begin{equation}\label{e4.15}
u_j>v_{s_0}\qquad\text{in }N_r(\Gamma).
\end{equation}


Decreasing \(s\) continuously, we claim that
\(u_j>v_s\) in \(N_r(\Gamma)\) for every
\(s\in[K\varepsilon_j,s_0]\).
Otherwise, by continuity and the strict boundary inequalities,
there would be a first contact at some
\(s_*\in[K\varepsilon_j,s_0)\).
Thus
\[
w:=u_j-v_{s_*}\geq0
\quad\text{on }\overline{N_r(\Gamma)},
\]
and \(w(x_*)=0\) at some \(x_*\in N_r(\Gamma)\).
By \eqref{e2.5}, there exists a symmetric measurable matrix
\(B\), with \(\lambda I\leq B\leq\Lambda I\), such that
\[
-\varepsilon_j^2\operatorname{div}_g(B\nabla w)+cw
=
\varepsilon_j^2\operatorname{div}_g
A(x,\nabla v_{s_*})-W'(v_{s_*})
>0,
\]
where
$
c=
\frac{W'(u_j)-W'(v_{s_*})}{u_j-v_{s_*}}
$
is defined by continuity for \(u_j=v_{s_*}\).
Since \(c\in L^\infty\), dividing by \(\varepsilon_j^2\)
and applying Lemma~\ref{lem:maximum-principles}(i)
gives a contradiction.


Let \(x\in\{u_j=0\}\cap N_{r/2}(\Gamma)\).
 Since \(u_j>v_{K\varepsilon_j}\), from \eqref{e4.15} we have
\[
0>v_{K\varepsilon_j}(x)
=
U\left(\frac{d_{K\varepsilon_j}(x)}{\varepsilon_j}\right)-\kappa K\varepsilon_j^{2},
\]
hence
\[
U\left(\frac{d_{K\varepsilon_j}(x)}{\varepsilon_j}\right)<\kappa K\varepsilon_j^{2}.
\]
Since \(U\) is differentiable at \(0\) and \(U'(0)>0\), by the implicit function theorem, $(U^{-1})'(0)=1/{U'(0)}.$ Hence
\[
U^{-1}(t)
=
U^{-1}(0)+(U^{-1})'(0)t+o(t)
=
\frac{t}{U'(0)}+o(t)
\qquad \text{as }t\to 0.
\]
Therefore, there exists a constant \(C>0\) such that
\[
\frac{d_{K\varepsilon_j}(x)}{\varepsilon_j}
<
U^{-1}(\kappa K\varepsilon_j^{2})
\leq
C\varepsilon_j^{2},
\]
and thus $d_{K\varepsilon_j}(x)\leq C\varepsilon_j^{3}.$  On the other side of $\Gamma$, use the hypersurfaces
\(\Gamma_{-s}\) and the distance functions \(-d_{-s}\).
Repeating the argument, and replacing \(u_j\) by \(-u_j\)
if necessary, gives
\[
d_{-K\varepsilon_j}(x)\geq-C\varepsilon_j^3.
\]
Here \(K\) may be enlarged so that the barrier construction
applies to both sides of $\Gamma$. Moreover, since  $
|d_{\pm K\varepsilon_j}(x)-d_0(x)|
\leq C\varepsilon_j,
$ we obtain
\[
-C\varepsilon_j-C\varepsilon_j^3
\leq d_0(x)
\leq C\varepsilon_j+C\varepsilon_j^3.
\]
Hence \(|d_0(x)|\leq C\varepsilon_j\). Finally,
since \(d_0=0\)   and
$
\nabla d_0=\frac{\nu_\Gamma}{F(x,\nu_\Gamma)}
\quad\text{on }\Gamma,
$
we conclude that
\[
\operatorname{dist}_g(x,\Gamma)
\leq C|d_0(x)|
\leq C\varepsilon_j.
\]
This proves \eqref{e4.14}.

\end{proof}

\section{Rigidity and Classification of Entire Solutions}\label{sec:slab}

Let $F_{0}:\mathbb{R}^{n}\setminus\{0\}\to[0,+\infty)$ be even, positively one-homogeneous  and satisfy \eqref{e1.6}.  Write $\Phi_{0}=F_{0}^{2}/2$ and $A_{0}=D\Phi_{0}$. We set \(F_0(0)=\Phi_0(0)=0\) and \(A_0(0)=0\).  We prove the entire-solution classification required below.  The argument is a finite-difference version of the  sliding methods in \cite{Farina1999,FarinaValdinoci2011}.

\begin{lemma}\label{lem:no-zero}
Let $v$ be a bounded entire weak solution of
\begin{equation}\label{e5.1}
 \operatorname{div}\bigl(A_{0}(\nabla v)\bigr)-W'(v)=0
 \qquad\text{in }\mathbb{R}^{n},
\end{equation}
with $|v|\leq1$.  If $v$ has no zeros, then $v\equiv1$ or $v\equiv-1$.
\end{lemma}

 \begin{proof}
 Without loss of generality, we assume \(v>0\) and set \(m:=\inf_{\mathbb R^n}v\).

If $m>0$, we claim that $m=1$.  Suppose to the contrary that $0<m<1$, we  choose a minimizing sequence \(\{x_k\}\) such that \(v(x_k)\to m\). For each \(k\), define $v_k(x):=v(x+x_k)\geq m$ for any $x\in\mathbb R^n$. By compactness and elliptic regularity for bounded entire solutions, there is a subsequence (still denoted by \(v_k\)) such that
\[
v_k\to v_\infty \quad\text{in }C^{1,\alpha'}_{\mathrm{loc}}(\mathbb R^n),
\]
where $v_\infty$ satisfies \eqref{e5.1} and $v_\infty(0)=m$. Let $w:=v_\infty-m\geq0$. Using \eqref{e2.5}, there is a symmetric measurable matrix \(B_{12}\) with \(\lambda I\le B_{12}\le \Lambda I\) such that
\[ -\operatorname{div}\left(B_{12}\nabla w\right)+c(x)w=m-m^3>0,\]
where $c(x)=\frac{W'(v_\infty)-W'(m)}{v_\infty-m}$ for $v_\infty\neq m$ and $c(x)=W''(m)$ for $v_\infty= m$ belongs to $L^\infty(\mathbb{R}^n)$. Thus \(w\ge0\) is a nonnegative strict supersolution that vanishes at the interior point \(0\), which is a contradiction with Lemma~\ref{lem:maximum-principles}(i).  Hence \(m=1\), i.e. \(v\equiv1\).

\medskip
Now assume that \(m=0\).
Choose \(x_k\) with \(v(x_k)\to0\), and set
\(v_k(x):=v(x+x_k)\).
As above, after passing to a subsequence,
\[
v_k\to v_\infty
\quad\text{in }C^{1,\alpha'}_{\mathrm{loc}}(\mathbb R^n),
\]
where \(v_\infty\geq0\) satisfies \eqref{e5.1} with \(v_\infty(0)=0\).
Using \eqref{e2.5}, we write its equation as
\[
-\operatorname{div}(B\nabla v_\infty)
+(v_\infty^2-1)v_\infty=0,
\]
where \(B\) is symmetric and measurable, with
\(\lambda I\leq B\leq\Lambda I\).
Lemma~\ref{lem:maximum-principles}(ii) gives
\(v_\infty\equiv0\).
Since \(v_k\to0\) locally uniformly, for every \(R>0\)
and all sufficiently large \(k\), we have
\[
0<v_k\leq\frac12
\quad\text{in }B_R.
\]
Consequently,
\[
-\operatorname{div}A_0(\nabla v_k)
=
v_k-v_k^3
\geq\frac34v_k
\quad\text{in }B_R.
\]

Taking the test function \(\varphi^2/v_k\) for any \(\varphi\in C_c^\infty(B_R)\), we obtain
\[
\int_{B_R} A_0(\nabla v_k) \cdot \nabla\left(\frac{\varphi^2}{v_k}\right) \, dx
\ge
\frac{3}{4}\int_{B_R} \varphi^2 \, dx.
\]
By the Convexity and two-homogeneity of $\Phi_0$, we have
\[\begin{split}
A_0(\nabla v_k) \cdot \nabla\left(\frac{\varphi^2}{v_k}\right)
=&\frac{2\varphi}{v_k} A_0(\nabla v_k)\cdot\nabla\varphi-\frac{\varphi^2}{v_k^2} A_0(\nabla v_k)\cdot\nabla v_k\\
\leq& \frac{2|\varphi|}{v_k}F_0(\nabla\varphi)F_0(\nabla v_k)-\frac{\varphi^2}{v_k^2}F_0(\nabla v_k)^2\\
=&F_0(\nabla\varphi)^2-\left(F_0(\nabla\varphi)^2- \frac{|\varphi|}{v_k}F_0(\nabla v_k)\right)^2
\leq F_0(\nabla\varphi)^2.
\end{split}\]
Therefore,
\[
\int_{B_R}F_0(\nabla\varphi)^2
\ge
\frac34\int_{B_R}\varphi^2\quad\text{for any }\varphi\in C_c^\infty(B_R).
\]
Fix \(\varphi_1\in C_c^\infty(B_1)\setminus\{0\}\) and set \(\varphi_R(x):=\varphi_1(x/R)\). Then
\[\frac{3}{4}\leq
\frac{\int_{B_R}F_0(\nabla\varphi_R)^2}{\int_{B_R}\varphi_R^2}
=
\frac{1}{R^2}
\frac{\int_{B_1}F_0(\nabla\varphi_1)^2}{\int_{B_1}\varphi_1^2}\to 0\quad\text{as }R\to\infty,
\]
which  is a contradiction. Thus \(m\neq0\).

Therefore, we conclude that \(m=1\), i.e. \(v\equiv1\). By a similar argument, we obtain \(v\equiv-1\) if \(v<0\).
\end{proof}

\begin{lemma}\label{lem:end-states}
Let $v$ be a solution of  \eqref{e5.1} with $|v|\leq1$, and suppose that
\begin{equation}\label{e5.4}
 \{v=0\}\subset\{|x_{n}|\leq L\}.
\end{equation}
Then there are $\sigma_{+},\sigma_{-}\in\{-1,1\}$ such that
\begin{equation}\label{e5.5}
 \lim_{x_{n}\to+\infty}v(x',x_{n})=\sigma_{+},\qquad
 \lim_{x_{n}\to-\infty}v(x',x_{n})=\sigma_{-},
\end{equation}
uniformly in $x'\in\R^{n-1}$.
\end{lemma}

\begin{proof}
Since \(\{v=0\}\subset\{|x_n|\le L\}\), the half-spaces \(\{x_n>L\}\) and \(\{x_n<-L\}\) contain no zeros of \(v\).
By continuity and connectedness, \(v\) has constant sign
on each of them.
Let \(\sigma_+\) and \(\sigma_-\)  be the signs of \(v\)
on these two half-spaces, respectively.

  Let \(x_k=(x'_k,x_{k,n})\) with \(x_{k,n}\to+\infty\), and set
\[
v_k(x):=v(x+x_k).
\]
By Lemma \ref{lem:compactness},  after passing to a subsequence,
\[
v_k\to v_\infty
\quad\text{in }C^1_{\mathrm{loc}}(\mathbb R^n),
\]
where \(v_\infty\) is an entire solution of \eqref{e5.1}. For every fixed \(R>0\), we have
\[
\{v_k=0\}\cap B_{R+2R_0}=\varnothing
\]
for all sufficiently large \(k\).
The nondegeneracy estimate therefore gives
\[
|v_k|\geq\eta_0
\quad\text{in }B_R.
\]
Passing to the limit and using the arbitrariness of \(R\), we obtain
\[
|v_\infty(x)|\geq\eta_0
\qquad\text{for all }x\in\mathbb R^n.
\]
Therefore, from  Lemma \ref{lem:no-zero}, we conclude that
\(v_\infty\equiv\sigma_+\).

Assume by contradiction that the convergence \(v(x',x_n)\to\sigma_+\) as \(x_n\to+\infty\) is not uniform in \(x'\). Then there exist \(\varepsilon>0\) and   a sequence
\(y_k=(y'_k,y_{k,n})\), with \(y_{k,n}\to+\infty\),   such that
\[
|v(y_k)-\sigma_+|\ge\varepsilon.
\]
Repeating the preceding argument, after passing to a subsequence,
\(v(x+y_k)\) converges in \(C^1_{\mathrm{loc}}(\mathbb R^n)\)
to \(w_\infty\equiv\sigma_+\).
Passing to the limit at the origin gives
\[
0
=
|w_\infty(0)-\sigma_+|
=
\lim_{k\to\infty}|v(y_k)-\sigma_+|
\geq\varepsilon,
\]
a contradiction.
Thus the convergence at \(+\infty\) is uniform in \(x'\).

The same argument with \(x_n\to-\infty\) gives the convergence
to \(\sigma_-\), uniformly in \(x'\).
This proves \eqref{e5.5}.

\end{proof}

\begin{lemma}\label{lem:end-comparison}
Let $\delta\in(0,1)$ be small such that
\[
 W''(t)\geq\kappa>0\quad\text{on}\quad[1-\delta,1]\cup[-1,-1+\delta].
\]
Let \(u_1,u_2\) be bounded weak solutions of \eqref{e5.1} in a neighborhood of the upper half-space \(\{x_n\ge T\}\) with
\begin{equation}\label{e0.2}u_1\ge u_2 \quad\text{on}\quad\{x_n=T\}.\end{equation}
Assume that either \(u_1,u_2\in[1-\delta,1]\), or \(u_1,u_2\in[-1,-1+\delta]\), and that
\begin{equation}\label{e0.3}
\liminf_{x_n\to+\infty}(u_1-u_2)\ge0
\end{equation}
uniformly in \(x'\). Then $u_1\ge u_2$ in $\{x_n>T\}.$
\end{lemma}

\begin{proof}
Set \(w=u_1-u_2\). By \eqref{e2.5}, there exist a symmetric measurable matrix \(B\) with \(\lambda I\le B\le\Lambda I\) and \(c\in L^\infty\) such that
\begin{equation}\label{e0.4}
-\operatorname{div}(B\nabla w)+c(x)w=0,
\end{equation}
where
$
c(x)=\int_0^1
W''\bigl(u_2(x)+t(u_1(x)-u_2(x))\bigr)\,\mathrm dt
$
and the choice of \(\delta\) gives \(c(x)\ge\kappa>0\).

Assume that \(\inf_{\{x_n>T\}}w=-\mu<0\).
Let \(x_k=(x_k',x_{k,n})\) be a minimizing sequence with
\(w(x_k)\to-\mu\).
By \eqref{e0.3}, there exists \(M>T\) such that
\[
w(x',x_n)\ge-\frac{\mu}{2}
\qquad\text{for }x_n\ge M.
\]
Thus \(x_{k,n}\le M\) for all sufficiently large \(k\).

Set
\(w^-=\max\{-w,0\}\). By \eqref{e0.4} and \(c\ge0\), we have
\[
-\operatorname{div}(B\nabla w^-)\le0
\qquad\text{in }\{x_n>T\}.
\]
Since \(w^-=0\) on \(\{x_n=T\}\), 
 the boundary decay
estimate gives constants \(C>0\) and \(\beta\in(0,1)\) such that
\[
w^-(x',T+t)\le Ct^\beta
\qquad\text{for }0<t<\frac12,
\]
uniformly in \(x'\).
Choose \(\rho\in(0,1/2)\) such that
\(C\rho^\beta<\mu/2\). Then, for all sufficiently large \(k\),
\[
T+\rho\le x_{k,n}\le M.
\]
Passing to a subsequence, we may assume that
\(x_{k,n}\to x_0\in[T+\rho,M]\).
Define
\[
w_k(x):=w(x+x_k)\quad\text{in}\quad \{x_n>T-x_{k,n}\}.
\]
Then
\[
w_k(0)=w(x_k)\to-\mu.
\]
Set
$
H_\infty:=\{x_n>T-x_0\}.
$
By Lemma~\ref{lem:compactness}, after passing to a subsequence,  \(w_k\to w_\infty\)   in \(C^{1}_{\mathrm{loc}}(H_\infty)\) to some \(w_\infty\), where $w_\infty$ satisfies \eqref{e0.4} with $w_\infty\ge-\mu.$ and
Moreover,
\[
w_\infty(0)=\lim_{k\to\infty}w_k(0)=-\mu.
\]
Thus \(w_\infty\) attains its minimum \(-\mu\) at the interior point \(0\).

Set \(\phi=w_\infty+\mu\ge0\). Then, there exists a  symmetric measurable matrix  \(B_\infty\)  with \(\lambda I\le B_\infty\le\Lambda I\) such that
\[
-\operatorname{div}(B_\infty\nabla\phi)+c_\infty\phi
=
-\operatorname{div}(B_\infty\nabla w_\infty)+c_\infty(w_\infty+\mu)
=
\mu c_\infty>0,
\]
where  \(c_\infty\ge\kappa>0\). This contradicts the strong maximum principle, Lemma~\ref{lem:maximum-principles} (i). Hence \(\inf w\ge0\), i.e. \(u_1\ge u_2\) in \(\{x_n>T\}\).
\end{proof}

\begin{lemma}\label{lem:opposite-ends}
Let $v$ be a solution of \eqref{e5.1} with $|v|\leq1$. If $\{v=0\}\neq\emptyset$ and \eqref{e5.4} holds, then
\[
 \sigma_{+}=-\sigma_{-}.
\]
\end{lemma}

\begin{proof}

Assume by contradiction that \(\sigma_+=\sigma_-\).
Since \(A_0\) and \(W'\) are odd, replacing \(v\) by
\(-v\) if necessary, we may assume that
\(\sigma_+=\sigma_-=1\).

 For $a\in\R$, define
\begin{equation}\label{e5.6}
 U_{a}(x)=U\left(\frac{x_{n}-a}{F_{0}(e_{n})}\right).
\end{equation}
The one-homogeneity of $A_{0}$ and Euler's identity give
\[
 \operatorname{div}A_{0}(\nabla U_{a})=U''=W'(U_{a}),
\]
which implies that $U_{a}$ is an entire solution of \eqref{e5.1}.

Choose \(\delta\) as in Lemma~\ref{lem:end-comparison} and \(M>L+1\) such that
\[
v\ge 1-\delta \qquad \text{whenever } |x_n|\ge M.
\]
We claim that
\begin{equation}\label{e0.5}
\inf_{\{x_n\le M\}} v > -1.
\end{equation}
Suppose to the contrary that there exists a sequence
\(x_k=(x_k',x_{k,n})\), with \(x_{k,n}\le M\), such that
\(v(x_k)\to-1\).
Since $\sigma_-=1$, we deduce that \(x_{k,n}\) is bounded below. Set $v_k(x):=v(x+x_k)$, then by Lemma \ref{lem:compactness}, after passing to a subsequence,
$
v_k\to v_\infty$
 in $C^1_{\mathrm{loc}}(\mathbb R^n)$,
where \(v_\infty\) is an entire solution of \eqref{e5.1}
satisfying
\[
v_\infty\ge-1,\qquad v_\infty(0)=-1.
\]
By \eqref{e2.5}, the function \(w=v_\infty+1\ge0\)
satisfies
\[
-\operatorname{div}(B_\infty\nabla w)
+v_\infty(v_\infty-1)w=0,
\]
where \(B_\infty\) is symmetric and measurable, with
\(\lambda I\le B_\infty\le\Lambda I\).
Since \(w(0)=0\), Lemma~\ref{lem:maximum-principles}(ii)
gives \(w\equiv0\), i.e., \(v_\infty\equiv-1\).
This is a contradiction with $\lim_{x_n\to\pm\infty}v_\infty(x',x_n)=1$.
Therefore, \eqref{e0.5} holds.


Since $U_{a}\to-1$ uniformly on $\{x_{n}\leq M\}$ as $a\to+\infty$, \eqref{e0.5} gives
\[
 U_{a}<v\qquad\text{on }\{x_{n}\leq M\}
\]
for all sufficiently large $a$. On the other hand, set
\[w:=v-U_a\quad\text{and}\quad \Omega_-:=\{x_n>M: w(x)<0\}.\]
Then,  $1-\delta\leq v<U_a\leq 1$ in $\Omega_-$ and there exists a bounded measurable symmetric matrix  $B$ with $\lambda I\leq B\leq\Lambda I$ such that
\[-\operatorname{div}(B\nabla w)+c(x) w=0,\]
where $c(x)=\frac{W'(v)-W'(U_a)}{v-U_a}\geq\kappa>0$ defined by continuity when \(v=U_a\).
In particular, \(c\ge\kappa>0\) in \(\Omega_-\). Then, by Lemma \ref{lem:maximum-principles} (iii),  $\Omega_-=\emptyset$ and $w\geq0$ in $\{x_n>M\}$. Therefore, we conclude that $U_{a}\leq v$ in $\mathbb{R}^{n}$ for all large $a$.

Set
\[
 a_{*}=\inf\{a\in\mathbb{R}:U_{a}\leq v\text{ in }\mathbb{R}^{n}\}.
\]
The  set is nonempty, so \(a_*<+\infty\).
We also have \(a_*>-\infty\).
 Indeed, if $a_{*}=-\infty$, then  \(U_a\le v\) for every \(a\in\mathbb R\). Letting \(a\to-\infty\), for any fixed \(x\in\mathbb R^n\) we have
\[
\frac{x_n-a}{F_0(e_n)}\to+\infty,
\]
and hence \(U_a(x)\to 1\). It follows that $v(x)\ge 1$ for every $x\in\mathbb R^n$. Since \(|v|\le 1\), we conclude \(v\equiv 1\), which is a contradiction with the assumption $\{v=0\}\neq\emptyset$.

We claim that
\begin{equation}\label{e5.8}
\inf_{\{x_n\le M\}}(v-U_{a_*})>0.
\end{equation}
Otherwise, there exists a sequence
\(x_k=(x_k',x_{k,n})\), with \(x_{k,n}\le M\), such that
\[
(v-U_{a_*})(x_k)\to0.
\]
The sequence \(x_{k,n}\) is bounded below.
Indeed, if a subsequence satisfied \(x_{k,n}\to-\infty\), then
$0=\lim_{k\to\infty}(v-U_{a_*})=1-(-1)=2$, this is a contradiction.  Passing to a subsequence, we may assume that
\(x_{k,n}\to b\in\mathbb R\).
Set \(v_k(x):=v(x+x_k)\).
By Lemma~\ref{lem:compactness}, after passing to a further
subsequence,
$
v_k\to\bar v$
 in $C^1_{\mathrm{loc}}(\mathbb R^n).$
Since
$
U_{a_*}(x+x_k)=U_{a_*-x_{k,n}}(x),
$
we obtain
\[
\bar v\ge U_{a_*-b},
\qquad
\bar v(0)=U_{a_*-b}(0).
\]
Applying \eqref{e2.5} and
Lemma~\ref{lem:maximum-principles}(ii) to
\(\bar v-U_{a_*-b}\), we obtain
$
\bar v\equiv U_{a_*-b}.
$
This is impossible because
$
1
=
\lim_{x_n\to-\infty}\bar v(x',x_n)
=
\lim_{x_n\to-\infty}U_{a_*-b}(x',x_n)
=
-1.
$
 Therefore,   \eqref{e5.8} holds.


Set
$
\varepsilon:=\frac12
\inf_{\{x_n\le M\}}(v-U_{a_*})>0.
$
The derivative $\partial_{a}U_{a}=-{U'\left(\frac{x_n-a}{F_0(e_n)}\right)}/{F_0(e_n)}$ is uniformly bounded.  Hence \eqref{e5.8} allows us to choose $\tau>0$ small enough  such that
\[
 U_{a_{*}-\tau}\leq U_{a_{*}}+\varepsilon\leq v\qquad\text{on}\quad\{x_{n}\leq M\}.
\]
Set $\bar\Omega_-:=\{x>M: v<U_{a_{*}-\tau}\}$, then $1-\delta\leq v<U_{a_{*}-\tau}\leq 1$. It follows from Lemma \ref{lem:maximum-principles} (iii) that $\bar\Omega_-=\emptyset$ and thus $v\geq U_{a_{*}-\tau}$ in $\{x>M\}$, which is a contradiction with the definition of $a_{*}$. Therefore, we conclude that $\sigma_+=-\sigma_-$.

\end{proof}

\begin{lemma}\label{lem:oblique}
Assume  that  $\sigma_+=1$ and $\sigma_-=-1$. Then, for every  $e\in\mathbb{R}^{n}$ with $e\cdot e_{n}>0$ and every $h>0$,
\begin{equation}\label{e5.12}
 v(x+he)\geq v(x)\qquad\forall x\in\mathbb{R}^{n}.
\end{equation}
Conversely, if $\sigma_+=-1$ and $\sigma_-=1$, then  $v(x+he)\leq v(x)$ for all $x\in\mathbb{R}^{n}$.
\end{lemma}

\begin{proof}
Fix \(e\in\mathbb R^n\) with \(e\cdot e_n>0\), and set
\[
 T_{h}v(x)=v(x+he)-v(x).
\]
Choose \(\delta\) as in Lemma~\ref{lem:end-comparison}
and \(M>0\) sufficiently large such that
\[
v\le-1+\delta\quad\text{for }x_n\le-M,
\qquad
v\ge1-\delta\quad\text{for }x_n\ge M.
\]

Let $he\cdot e_{n}>2M$, and set $w(x):=v(x+he)-v(x)$. Then, there exists measurable symmetric matrix $B$ with $\lambda I\leq B\leq\Lambda I$ such that
\[-\operatorname{div}(B\nabla w)+c(x)w=0\quad\text{in}\quad\{w<0\},\]
where $c(x)=\frac{W'(v(x+he))-W'(v(x))}{v(x+he)-v(x)}$ defined by continuity when \(v(x+he)=v(x)\). If \(x_n\le-M\), then on \(\{w<0\}\) we have
\[
-1\le v(x+he)<v(x)\le-1+\delta.
\]
If \(x_n>-M\), then
$
x_n+he\cdot e_n>-M+2M=M,
$
and hence, on \(\{w<0\}\),
\[
1-\delta\le v(x+he)<v(x)\le1.
\]
Therefore, the mean value theorem gives
\(c(x)\ge\kappa>0\) on \(\{w<0\}\).
By Lemma~\ref{lem:maximum-principles}(iii), we obtain
\(\{w<0\}=\emptyset\). Thus
\begin{equation}\label{e0.6}
T_hv\ge0
\qquad\text{if }he\cdot e_n>2M.
\end{equation}

Let
\[
 \mathcal{S}
 =
 \left\{
 h>0:
 T_{\tau}v\geq0\text{ on }\{|x_{n}|\leq M\}
 \text{ for every }\tau\geq h
 \right\}.
\]
By \eqref{e0.6}, \(\mathcal S\ne\emptyset\).
Moreover, if \(h\in\mathcal S\) and \(h'\ge h\), then
\(h'\in\mathcal S\).
Set \(h_0=\inf\mathcal S\).
By the definition of \(h_0\) and the continuity of \(v\),
\[
T_\tau v\ge0
\qquad\text{on }\{|x_n|\le M\}
\quad\text{for every }\tau\ge h_0.
\]
Let \(w_0:=T_{h_0}v\), then $\{w_0<0\}\subset\{|x_n|>M\}$.
 Applying Lemma \ref{lem:maximum-principles} (iii) to $w_0$ and arguing as before, we deduce that
\[
T_{h_0}v\ge0\qquad\text{in }\mathbb R^n.
\]

Suppose that \(h_0>0\).
Choose \(s_k\in(0,h_0)\) with \(s_k\to h_0\).
Since \(s_k\notin\mathcal S\), while
\(T_\tau v\ge0\) on \(\{|x_n|\le M\}\) for every
\(\tau\ge h_0\), there exist
\(\tau_k\in[s_k,h_0)\) and
\(x_k=(x_k',x_{k,n})\), with \(|x_{k,n}|\le M\), such that
\[
T_{\tau_k}v(x_k)<0.
\]
In particular, \(\tau_k\to h_0\).  Set
\[
v_k(x):=v(x'+x_k',x_n).
\]
By Lemma~\ref{lem:compactness}, after passing to a subsequence,
\(v_k\) converges to an entire solution \(\bar v\) in
\(C^1_{\mathrm{loc}}(\mathbb R^n)\) with
\[
\lim_{x_n\to-\infty}\bar v(x',x_n)=-1,
\qquad
\lim_{x_n\to+\infty}\bar v(x',x_n)=1
\]
uniformly in \(x'\).

The uniform interior gradient estimate gives
$
\|\nabla v\|_{L^\infty(\mathbb R^n)}\le C.
$
Consequently,
\[
0\le T_{h_0}v(x_k)
\le T_{\tau_k}v(x_k)+C|e|(h_0-\tau_k)
<C|e|(h_0-\tau_k)\to0.
\]
Therefore,
\[
v_k((0,x_{k,n})+h_0e)-v_k(0,x_{k,n})
=
v(x_k+h_0e)-v(x_k)
\to0.
\]
After passing to a further subsequence, we may assume that
\(x_{k,n}\to x_{0,n}\in[-M,M]\).
Set \(x_0=(0,x_{0,n})\). Then
$
\bar v(x_0+h_0e)-\bar v(x_0)=0.
$
Moreover, since \(T_{h_0}v\ge0\) in \(\mathbb R^n\),
$
\bar w(x):=\bar v(x+h_0e)-\bar v(x)\ge0
\qquad\text{in }\mathbb R^n.
$
By \eqref{e2.5} and
Lemma~\ref{lem:maximum-principles}(ii), the equality
\(\bar w(x_0)=0\) implies that \(\bar w\equiv0\).
Hence
\[
\bar v(x+h_0e)=\bar v(x)
\qquad\text{for all }x\in\mathbb R^n.
\]
It follows that, for every \(m\in\mathbb Z\),
\[
\bar v(x+mh_0e)=\bar v(x).
\]
Since \(h_0e\cdot e_n>0\) and the limits at both ends
are uniform in \(x'\), we obtain
\[
\bar v(x)
=\lim_{m\to+\infty}\bar v(x+mh_0e)=1,
\qquad
\bar v(x)
=\lim_{m\to-\infty}\bar v(x+mh_0e)=-1,
\]
a contradiction.


Therefore \(h_0=0\), and every \(h>0\) belongs to
\(\mathcal S\). In particular,
\[
v(x+he)\ge v(x)
\qquad\text{on }\{|x_n|\le M\}.
\]
Fix \(h>0\) and set \(w:=v(x+he)-v(x)\).
On \(\{x_n>M\}\cap\{w<0\}\), we have
\[
1-\delta\le v(x+he)<v(x)\le1.
\]
Since \(w\ge0\) on \(\{x_n=M\}\),
By Lemma~\ref{lem:maximum-principles}(iii), we obtain \(w\ge0\) in
\(\{x_n>M\}\).

\noindent Similarly, on \(\{x_n<-M\}\cap\{w<0\}\),
\[
-1\le v(x+he)<v(x)\le-1+\delta.
\]
Since \(w\ge0\) on \(\{x_n=-M\}\), the same comparison
principle gives \(w\ge0\) in \(\{x_n<-M\}\).
Thus
\[
v(x+he)\ge v(x)
\qquad\text{for all }x\in\mathbb R^n,\quad h>0.
\]

If instead \(\sigma_+=-1\) and \(\sigma_-=1\), then
\(-v\) also solves \eqref{e5.1}, since \(A_0\) and \(W'\)
are odd. Applying the preceding argument to \(-v\) gives
\(v(x+he)\le v(x)\).


\end{proof}

\begin{proposition}\label{thm:slab}
Let \(v\) be a bounded entire weak solution of \eqref{e5.1}
with \(|v|\le1\), $\{v=0\}\neq\emptyset$ and satisfies
\eqref{e5.4}. Then, after possibly replacing \(v\) by \(-v\),
there is \(a\in[-L,L]\) such that
\begin{equation}\label{e5.9}
v(x)=U\left(\frac{x_n-a}{F_0(e_n)}\right).
\end{equation}
\end{proposition}

\begin{proof}
By Lemmas~\ref{lem:end-states} and \ref{lem:opposite-ends},
after replacing \(v\) by \(-v\) if necessary, we may assume
that \(\sigma_+=1\) and \(\sigma_-=-1\).

Let \(\zeta\perp e_n\) and \(h>0\).
For \(\delta>0\), applying \eqref{e5.12} to
\(\zeta+\delta e_n\), we obtain
\[
v(x+h(\zeta+\delta e_n))\ge v(x).
\]
Letting \(\delta\to0\), the continuity of \(v\) yields
\[
v(x+h\zeta)\ge v(x).
\]
Replacing \(\zeta\) by \(-\zeta\), we obtain
$
v(y-h\zeta)\ge v(y)$
 for every $y\in\mathbb R^n.$
Taking \(y=x+h\zeta\), we get
\[
v(x)\ge v(x+h\zeta).
\]
Therefore,
$
v(x+h\zeta)=v(x)$
 for all $\zeta\perp e_n,\quad h>0,$
and hence
\[
v(x)=f(x_n).
\]


Substituting this into \eqref{e5.1} and using $A_0(e_n)\cdot e_n=F_0(e_n)^2$, we deduce that
\begin{equation}\label{e5.14}
F_0(e_n)^2 f''=W'(f).
\end{equation}
Taking \(e=e_n\) in \eqref{e5.12} shows that $f(x_n+h)\geq f(x_n)$, i.e. \(f\) is nondecreasing. Multiplying \eqref{e5.14} by \(f'\) and integrating gives
\[
\frac{F_0(e_n)^2}{2}(f')^2-W(f)=C.
\]
Since  \(f\) is nondecreasing  and \(\lim_{x_n\to\pm\infty}f=\pm1\), we have
$
f'\ge0$ and $
\int_{\mathbb R}f'(t)\,\mathrm dt=2.$
Moreover, \eqref{e5.14} and \(|f|\le1\) imply that
\(f''\) is bounded. Thus \(f'\) is uniformly continuous,
and consequently \(f'(t)\to0\) as \(t\to\pm\infty\).
Hence \(C=0\), and
\[
\frac{F_0(e_n)^2}{2}(f')^2=W(f).
\]

Let \(a\) be the unique zero of \(f\).
By separation of variables, we obtain
\[
v(x)=U\left(\frac{x_n-a}{F_0(e_n)}\right).
\]
Finally, \(\{v=0\}=\{x_n=a\}\), so \eqref{e5.4}
implies \(a\in[-L,L]\).
This proves \eqref{e5.9}.


\end{proof}

\section{Proof of Theorem \ref{thm:main}}\label{sec:energy}

Decrease \(\rho\) if necessary so that
\(N_{2\rho}(\Gamma)\subset\subset N_{r/2}(\Gamma)\),
where \(r\) is as in Proposition~\ref{prop:trapping},
and the normal exponential map is a diffeomorphism on
\(\Gamma\times(-2\rho,2\rho)\).
Fix a unit normal \(\nu_\Gamma\) and define
\(\Psi:\Gamma\times(-2\rho,2\rho)\to N_{2\rho}(\Gamma)\) by
\begin{equation}\label{e6.1}
\Psi(p,t)=\exp_p(t\nu_\Gamma(p)).
\end{equation}
Let \(J(p,t)\) be the corresponding volume Jacobian, so that
\[
\mathrm dV_g
=
J(p,t)\,\mathrm d\mathcal H_g^{n-1}(p)\,\mathrm dt,
\qquad J(p,0)=1.
\]


\begin{proposition}\label{prop:blowup}
Let $u_j$, $\varepsilon_j$ and $\Gamma$ be as in Theorem \ref{thm:main} and let $p_{j}\in\Gamma$ and $p_{j}\to p$. Define
\begin{equation}\label{e6.2}
 v_{j}(y)=u_{j}\bigl(\exp_{p_{j}}(\varepsilon_{j}y)\bigr),
\end{equation}
where \(\exp_{p_j}\) is the exponential map at \(p_j\).
After passing to a subsequence, $v_{j}$ converges locally in $C^{1,\alpha'}$ to an entire solution of
\begin{equation}\label{e6.3}
 \operatorname{div}\bigl(A(p,\nabla v)\bigr)-W'(v)=0
 \qquad\text{in }T_{p}M\simeq\R^{n}.
\end{equation}
The limit is either $1$, $-1$, or
\begin{equation}\label{e6.4}
 \pm U\left(\frac{y\cdot\nu_{\Gamma}(p)-a}{F(p,\nu_{\Gamma}(p))}\right)
\end{equation}
for a bounded constant $a$.
\end{proposition}

\begin{proof}

 Choose a smooth orthonormal frame near \(p\) such that
\(e_n=\nu_\Gamma\) on \(\Gamma\).

Set
\[
g_j(y):=(\exp_{p_j}^*g)_{\varepsilon_j y},
\qquad
P_j(y):=\mathrm d(\exp_{p_j})_{\varepsilon_j y},
\]
and define
\[
A_j(y,\xi)
:=
P_j(y)^{-1}
A\bigl(\exp_{p_j}(\varepsilon_j y),P_j(y)\xi\bigr).
\]
For each fixed \(R>0\) and all sufficiently large \(j\),
\[
\operatorname{div}_{g_j}
A_j(y,\nabla_{g_j}v_j)-W'(v_j)=0
\qquad\text{in }B_R.
\]

In these normal coordinates,
\[
(g_j)_{kl}(y)
=
\delta_{kl}+O(\varepsilon_j^2|y|^2),
\]
and \(g_j\) converges to the Euclidean metric in
\(C^1_{\mathrm{loc}}\).
By \eqref{eq:base-A}, for every fixed \(R>0\) and all sufficiently large \(j\),
\[
|A_j(y,\xi)-A(p,\xi)|
\le
C_R\bigl(d_g(p_j,p)+\varepsilon_j\bigr)|\xi|
\qquad\text{for }|y|\le R.
\]
Thus \(A_j(y,\xi)\to A(p,\xi)\) locally uniformly.

In Euclidean coordinates, the rescaled equation  is
\[
\operatorname{div}\left(
\sqrt{\det g_j}\,
A_j(y,g_j^{-1}\nabla v_j)
\right)
-
\sqrt{\det g_j}\,W'(v_j)=0
\]
with \(|v_j|\le1\).
The local regularity estimates used in
Lemma~\ref{lem:compactness} therefore give uniform
\(C^{1,\alpha}\) bounds.
After passing to a subsequence,
\[
v_j\to v
\qquad\text{in }C^{1,\alpha'}_{\mathrm{loc}}(\mathbb R^n)
\quad\text{for every }0<\alpha'<\alpha.
\]
Passing to the limit in the weak equation, we obtain
\eqref{e6.3}, with \(|v|\le1\).

By Proposition~\ref{prop:trapping}, there exists \(C_0>0\)
such that, for all sufficiently large \(j\),
\[
\{u_j=0\}\cap N_{r/2}(\Gamma)
\subset N_{C_0\varepsilon_j}(\Gamma).
\]
A Taylor expansion of the signed distance to \(\Gamma\)
gives
\[
\frac{\operatorname{dist}_g
(\exp_{p_j}(\varepsilon_j y),\Gamma)}
{\varepsilon_j}
=
|y_n|+O(\varepsilon_j|y|^2)
\]
locally uniformly in \(y\).
Set \(L=C_0+2R_0+1\).
For any fixed \(y\) with \(|y_n|>L\), the preceding
estimate and Proposition~\ref{prop:trapping} imply that,
for all sufficiently large \(j\),
\[
B_{2R_0\varepsilon_j}
\bigl(\exp_{p_j}(\varepsilon_j y)\bigr)
\cap\{u_j=0\}=\emptyset.
\]
Lemma~\ref{lem:nondegeneracy} then gives
$
|v_j(y)|\ge\eta_0
$ for large \(j\).
Passing to the limit yields \(|v(y)|\ge\eta_0\).
Consequently,
\[
\{v=0\}\subset\{|y_n|\le L\}.
\]
If \(\{v=0\}=\emptyset\), Lemma~\ref{lem:no-zero}
gives \(v\equiv1\) or \(v\equiv-1\).
If \(\{v=0\}\ne\emptyset\), then
Proposition~\ref{thm:slab} gives
\[
v(y)
=
\pm U\left(\frac{y_n-a}{F(p,e_n)}\right)
=
\pm U\left(
\frac{y\cdot\nu_\Gamma(p)-a}{F(p,\nu_\Gamma(p))}
\right),
\]
with \(|a|\le L\).
This proves \eqref{e6.4}.


\end{proof}

We now use Proposition \ref{prop:blowup} to estimate the energy in a tubular neighborhood of \(\Gamma\).
For $R>0$ and $p\in\Gamma$, define
\begin{equation}\label{e6.5}
 I_{j}^{R}(p)=\int_{-R}^{R}\Bigg[\frac{1}{2}F\left(\Psi(p,\varepsilon_{j}s),
\varepsilon_{j}\nabla u_{j}(\Psi(p,\varepsilon_{j}s))\right)^{2}
 +W\bigl(u_{j}(\Psi(p,\varepsilon_{j}s))\bigr)\Bigg]J(p,\varepsilon_{j}s)\,\mathrm{d}s.
\end{equation}
Then, the anisotropic Allen-Cahn energy $E_{\varepsilon_j,F}(u_j)$ in $N_{R\varepsilon_j}(\Gamma)$ can be rewritten as
\[E_{\varepsilon_j,F}(u_j;N_{R\varepsilon_j}(\Gamma))= \int_{N_{R\varepsilon_j}(\Gamma)} \left( \frac{\varepsilon_j}{2}F(x,\nabla u_j)^2 + \frac{W(u_j)}{\varepsilon_j} \right) \mathrm{d}V_g= \int_\Gamma I_j^R(p) \,\mathrm{d}\mathcal{H}_g^{n-1}(p). \]


\begin{lemma}\label{lem:one-layer}
Let $u_j$, $\varepsilon_j$ and $\Gamma$ be as in Theorem \ref{thm:main}. For every \(R>0\),
\begin{equation}\label{e6.7}
\limsup_{j\to\infty}E_{\varepsilon_j,F}
\bigl(u_j;N_{R\varepsilon_j}(\Gamma)\bigr)
\le \sigma_W \mathcal{A}_F(\Gamma).
\end{equation}
\end{lemma}

\begin{proof}
Fix \(R>0\). We first show that
\begin{equation}\label{e6.6}
\limsup_{j\to\infty}\sup_{p\in\Gamma}
\left(
I_j^R(p)-\sigma_WF(p,\nu_\Gamma(p))
\right)\le0.
\end{equation}
Assume by contradiction that
\[
\limsup_{j\to\infty}\sup_{p\in\Gamma}
\left(
I_j^R(p)-\sigma_WF(p,\nu_\Gamma(p))
\right)>0.
\]
By the compactness of \(\Gamma\), after passing to a subsequence,
there exist \(\varepsilon>0\) and \(p_j\in\Gamma\) with
\(p_j\to p\in\Gamma\) such that
\begin{equation}\label{e0.9}
I_j^R(p_j)-\sigma_WF(p_j,\nu_\Gamma(p_j))
>\frac{\varepsilon}{2}
\end{equation}
for every \(j\).

 By Proposition~\ref{prop:blowup}, after passing to a further
subsequence, the rescaled functions
\[
v_j(y)=u_j\bigl(\exp_{p_j}(\varepsilon_j y)\bigr)
\]
converge locally in \(C^{1,\alpha'}\) to an entire solution
\(v\) of
\[
\operatorname{div}\bigl(A(p,\nabla v)\bigr)-W'(v)=0.
\]
Moreover, \(v\) is either \(1\), \(-1\), or
\[
\pm U\left(
\frac{y\cdot\nu_\Gamma(p)-a}{F(p,\nu_\Gamma(p))}
\right)
\]
for some bounded constant \(a\).

 As in Proposition~\ref{prop:blowup}, write
$
g_j(y)=(\exp_{p_j}^*g)_{\varepsilon_j y}.
$
Then
\[
\varepsilon_j\nabla u_j
\bigl(\exp_{p_j}(\varepsilon_j y)\bigr)
=
\mathrm d(\exp_{p_j})_{\varepsilon_j y}
\bigl(\nabla_{g_j}v_j(y)\bigr).
\]
Since
\[
\Psi(p_j,\varepsilon_j s)
=
\exp_{p_j}\bigl(\varepsilon_j s\nu_\Gamma(p_j)\bigr),
\qquad
J(p_j,\varepsilon_j s)\longrightarrow1
\]
uniformly for \(s\in[-R,R]\), the local \(C^1\) convergence
and the convergence of the metrics give
$
I_j^R(p_j)\longrightarrow I_\infty^R(p),
$
where
\[
I_\infty^R(p)
=
\int_{-R}^{R}
\left[
\frac12 F\bigl(p,\nabla v(s\nu_\Gamma(p))\bigr)^2
+
W\bigl(v(s\nu_\Gamma(p))\bigr)
\right]\,\mathrm ds.
\]


If \(v\equiv 1\) or \(v\equiv -1\), then \(\nabla v=0\) and \(W(v)=0\), hence $I_\infty^R(p)=0.$

If \(v(y)=\pm U\left(\frac{y\cdot\nu_\Gamma(p)-a}{F(p,\nu_\Gamma(p))}\right)\),  then $v(s\,\nu_\Gamma(p))=\pm U\left(\frac{s-a}{F(p,\nu_\Gamma(p))}\right).$
Since \(F\) is even and one-homogeneous and \(U'>0\),
\[
F\bigl(p,\nabla v(s\,\nu_\Gamma(p))\bigr)=U'\left(\frac{s-a}{F(p,\nu_\Gamma(p))}\right).
\]
Under the change of variables \(t=(s-a)/F(p,\nu_\Gamma(p))\), we obtain
\[
I_\infty^R(p)
=
F(p,\nu_\Gamma(p))
\int_{t_1}^{t_2}
\left[
\frac12 U'(t)^2+W(q(t))
\right]\mathrm dt
\le
\sigma_W F(p,\nu_\Gamma(p)),
\]
where $\sigma_W=\int_{-\infty}^{\infty}\left[\frac12 U'(t)^2+W(U(t))\right]\mathrm dt$.
In conclusion, we have
\[
I_\infty^R(p)\le \sigma_W F(p,\nu_\Gamma(p)).
\]
Consequently,
\[
\limsup_{j\to\infty}
\left(I_j^R(p_j)-\sigma_W F(p_j,\nu_\Gamma(p_j))\right)\le 0,
\]
contradicting with \eqref{e0.9}. Hence \eqref{e6.6} holds.

Finally,
\[
\begin{aligned}
&E_{\varepsilon_j,F}(u_j;N_{R\varepsilon_j}(\Gamma))
-\sigma_W\mathcal A_F(\Gamma)\\
&\quad=
\int_\Gamma
\left(
I_j^R(p)-\sigma_WF(p,\nu_\Gamma(p))
\right)\,\mathrm d\mathcal H_g^{n-1}(p)\\
&\quad\le
\mathcal H_g^{n-1}(\Gamma)
\sup_{p\in\Gamma}
\left(
I_j^R(p)-\sigma_WF(p,\nu_\Gamma(p))
\right).
\end{aligned}
\]
Taking the upper limit as \(j\to\infty\) and using
\eqref{e6.6}, we obtain \eqref{e6.7}.

\end{proof}

\begin{lemma}\label{lem:tail}
Let $u_j$, $\varepsilon_j$ and $\Gamma$ be as in Theorem \ref{thm:main}.
There are $C,\gamma>0$ such that, for every  $R$ sufficiently large,
\begin{equation}\label{e6.8}
 \limsup_{j\to\infty}E_{\varepsilon_{j},F}
 \left(u_{j};N_{\rho}(\Gamma)\setminus N_{R\varepsilon_{j}}(\Gamma)\right)
 \leq Ce^{-2\gamma R}.
\end{equation}
\end{lemma}

\begin{proof}
By Proposition~\ref{prop:trapping},
\[
\{u_j=0\}\cap N_{2\rho}(\Gamma)
\subset N_{C_0\varepsilon_j}(\Gamma).
\]
We take \(\rho\) small enough that normal coordinates
\(x=\Psi(p,t)\) are defined on \(N_{2\rho}(\Gamma)\), with
\[
|t|=\operatorname{dist}_g(x,\Gamma).
\]
For all sufficiently large \(j\), we have
\[
\operatorname{dist}_g(x,\{u_j=0\})
\ge \frac12\bigl(|t|-C_0\varepsilon_j\bigr)_+
\qquad
\text{for }x\in N_{\rho+2\varepsilon_j}(\Gamma).
\]
Indeed, if \(z\in\{u_j=0\}\cap N_{2\rho}(\Gamma)\), then
\[
\operatorname{dist}_g(x,z)
\ge
\operatorname{dist}_g(x,\Gamma)
-\operatorname{dist}_g(z,\Gamma)
\ge |t|-C_0\varepsilon_j.
\]
If \(z\in\{u_j=0\}\setminus N_{2\rho}(\Gamma)\), then
\[
\operatorname{dist}_g(x,z)
\ge 2\rho-|t|
\ge \frac12|t|
\]
for all sufficiently large \(j\), since
\(|t|<\rho+2\varepsilon_j\).


Since \(\Gamma\) is connected, \(u_j\) has a constant sign
\(s_\pm\in\{-1,1\}\) on each of the sets
\[
\{x\in N_{2\rho}(\Gamma):
C_0\varepsilon_j<\pm t<2\rho\}.
\]
The signs \(s_\pm\) may depend on \(j\).

Fix \(R>0\) sufficiently large.
For every integer \(k\ge\lfloor R\rfloor\) with
\(k\varepsilon_j<\rho\), define
\[
\mathcal A_{j,k}
=
\{x\in N_\rho(\Gamma):
k\varepsilon_j<|t|<(k+1)\varepsilon_j\},
\]
and
\[
\mathcal A_{j,k}^{*}
=
\{x\in N_{2\rho}(\Gamma):
(k-1)\varepsilon_j<|t|<(k+2)\varepsilon_j\}.
\]
The annuli \(\mathcal A_{j,k}\) are mutually disjoint
and cover
\(N_\rho(\Gamma)\setminus N_{R\varepsilon_j}(\Gamma)\)
up to a set of measure zero.
Moreover,
$
\mathcal A_{j,k}^{*}
\subset N_{\rho+2\varepsilon_j}(\Gamma).
$


For \(k>C_0+1+4R_0\), the preceding distance estimate gives
\[
\operatorname{dist}_g(x,\{u_j=0\})
\ge \frac12(k-1-C_0)\varepsilon_j
>2R_0\varepsilon_j
\qquad
\text{on }\mathcal A_{j,k}^{*}.
\]
Hence, by Lemma~\ref{lem:pure-decay}, there are
\(C,\gamma>0\), independent of \(j\) and \(k\), such that
\[
|1-s_\pm u_j|
\le Ce^{-\gamma(k-1-C_0)}
\qquad
\text{on }\mathcal A_{j,k}^{*}\cap\{\pm t>0\}.
\]
Choose a smooth cutoff \(\eta\), supported in
\(\mathcal A_{j,k}^{*}\), such that
\[
0\le\eta\le1,\qquad
\eta=1\text{ on }\mathcal A_{j,k},\qquad
|\nabla\eta|_g\le \frac{C}{\varepsilon_j}.
\]
 Applying \eqref{e3.7} and using that \(\operatorname{Vol}_g(\mathcal A_{j,k}^{*})\leq C \varepsilon_j\mathcal H_g^{n-1}(\Gamma)\), we obtain
\[
E_{\varepsilon_j,F}(u_j;\mathcal A_{j,k})
\leq C  \varepsilon_j^{-1}  e^{-2\gamma(k-1-C_0)}\varepsilon_jH_g^{n-1}(\Gamma)
= C_1 e^{-2\gamma(k-1-C_0)},
\]
where \(C\) is independent of \(j\), \(k\), and \(R\).
Taking \(R\) sufficiently large that
\(\lfloor R\rfloor>C_0+1+4R_0\), and summing over the
annuli, we obtain
\[
E_{\varepsilon_j,F}
\bigl(u_j;N_\rho(\Gamma)\setminus
N_{R\varepsilon_j}(\Gamma)\bigr)
 \le
\sum_{\substack{k\ge\lfloor R\rfloor
k\varepsilon_j<\rho}}
E_{\varepsilon_j,F}(u_j;\mathcal A_{j,k})
 \le
C\sum_{k\ge\lfloor R\rfloor}
e^{-2\gamma(k-1-C_0)}
\le Ce^{-2\gamma R}.
\]
Passing to the upper limit as \(j\to\infty\), we obtain
\eqref{e6.8}.

\end{proof}

Now, we provide the proof of Theorem~\ref{thm:main}.

\begin{proof}[Proof of Theorem~\ref{thm:main}]
For every fixed \(R>0\) sufficiently large, combining
Lemmas~\ref{lem:one-layer} and~\ref{lem:tail}, we obtain
\[
\limsup_{j\to\infty}
E_{\varepsilon_j,F}(u_j;N_\rho(\Gamma))
\le
\sigma_W\mathcal A_F(\Gamma)+Ce^{-2\gamma R}.
\]
Letting \(R\to\infty\), we conclude that
\begin{equation}\label{e0.11}
\limsup_{j\to\infty}
E_{\varepsilon_j,F}(u_j;N_\rho(\Gamma))
\le \sigma_W\mathcal A_F(\Gamma).
\end{equation}

Since \(N_\rho(\Gamma)\subset\subset N_{2\rho}(\Gamma)\) and
the limiting energy measure in \(N_{2\rho}(\Gamma)\)
is supported on \(\Gamma\), we have  \(\partial N_\rho(\Gamma)\cap\Gamma=\varnothing\).
Thus \eqref{e1.11} gives
\begin{equation}\label{e0.12}
\lim_{j\to\infty}
E_{\varepsilon_j,F}(u_j;N_\rho(\Gamma))
=
m\sigma_W\mathcal A_F(\Gamma).
\end{equation}
From \eqref{e0.11}, \eqref{e0.12}, and
\(\sigma_W\mathcal A_F(\Gamma)>0\), we deduce \(m\le1\).
Since \(m\) is a positive integer, it follows that \(m=1\).
\end{proof}

\subsection*{Acknowledgements}
The authors are grateful to Dr. Ling Wang for telling us this problem and giving some suggestion for this work.

\end{document}